\documentclass[11pt]{article}

\usepackage{amsmath, amssymb, amsthm, verbatim,enumerate,bbm,color} 
\numberwithin{equation}{section}
\usepackage{float}

\usepackage{mathrsfs}

\usepackage{tikz}
\usetikzlibrary{calc}
\usetikzlibrary{math}
\usetikzlibrary{decorations,decorations.markings,decorations.pathreplacing,decorations.pathmorphing}
\usetikzlibrary{fadings}
\usetikzlibrary{fit}
\usetikzlibrary{matrix}
\usetikzlibrary{patterns}
\usetikzlibrary{positioning}
\usetikzlibrary{shadows}
\usetikzlibrary{shapes,shapes.geometric}
\usetikzlibrary{trees}
\usetikzlibrary{overlay-beamer-styles}
\usepackage{tikzsymbols}

\usepackage[colorlinks,citecolor=blue,bookmarks=true]{hyperref}

\date{\today}
\allowdisplaybreaks

\theoremstyle{plain}
\newtheorem{theorem}{Theorem}[section]
\newtheorem{lemma}[theorem]{Lemma}
\newtheorem{claim}[theorem]{Claim}
\newtheorem{proposition}[theorem]{Proposition}

\newtheorem{conjecture}[theorem]{Conjecture}

\newtheorem{remark}[theorem]{Remark}
\newtheorem{definition}[theorem]{Definition}

\makeatletter
\def\moverlay{\mathpalette\mov@rlay}
\def\mov@rlay#1#2{\leavevmode\vtop{%
		\baselineskip\z@skip \lineskiplimit-\maxdimen
		\ialign{\hfil$\m@th#1##$\hfil\cr#2\crcr}}}
\newcommand{\charfusion}[3][\mathord]{
	#1{\ifx#1\mathop\vphantom{#2}\fi
		\mathpalette\mov@rlay{#2\cr#3}
	}
	\ifx#1\mathop\expandafter\displaylimits\fi}
\makeatother

\makeatletter
\renewenvironment{proof}[1][\proofname]
{\par\pushQED{\qed}
	\normalfont\topsep6\p@\@plus6\p@\relax\trivlist
	\item[\hskip\labelsep\bfseries#1\@addpunct{.}]
	\ignorespaces}
{\popQED\endtrivlist\@endpefalse}

\makeatother

\newcommand{\N}{\mathcal N}

\newcommand{\J}{\mathcal J}

\renewcommand{\P}{\mathcal P}

\RequirePackage{color}\definecolor{RED}{rgb}{1,0,0}\definecolor{BLUE}{rgb}{0,0,1} 

\newcommand{\ex}{\textup{ex}}
\newcommand{\exbar}{\overline{\text{ex}}}

\newcommand{\poly}{\text{poly}}
\newcommand{\rem}{\mathrm{rem}}

\usepackage{hyperref}

\begin{document}

\title{Hardness of Hypergraph Edge Modification Problems}

\author{Lior Gishboliner\thanks{Department of Mathematics, University of Toronto. Email: lior.gishboliner@utoronto.ca.} \and Yevgeny Levanzov\thanks{School of Mathematics, Tel Aviv University, Tel Aviv 69978, Israel. Email: yevgenyl@mail.tau.ac.il.} \and Asaf Shapira\thanks{School of Mathematics, Tel Aviv University, Tel Aviv 69978, Israel. Email: asafico@tau.ac.il. Supported in part
by ERC Consolidator Grant 863438 and NSF-BSF Grant 20196.}}

\date{}
\maketitle

\begin{abstract}
For a fixed graph $F$, let $\ex_F(G)$ denote the size of the largest $F$-free subgraph of $G$.
Computing or estimating $\ex_F(G)$ for various pairs $F,G$ is one of the central problems in extremal combinatorics.
It is thus natural to ask how hard is it to compute this function.
Motivated by an old problem of Yannakakis from the 80's, Alon, Shapira and Sudakov [ASS'09] proved that for every non-bipartite
graph $F$, computing $\ex_F(G)$ is NP-hard.
Addressing a conjecture of Ailon and Alon (2007), we prove a hypergraph analogue of this theorem,
showing that for every $k \geq 3$ and every non-$k$-partite $k$-graph $F$, computing $\ex_F(G)$ is NP-hard.
Furthermore, we conjecture that our hardness result can be extended to all $k$-graphs $F$ other than a matching of fixed size.
If true, this would give a precise characterization of the $k$-graphs $F$ for which computing $\ex_F(G)$ is NP-hard,
since we also prove that when $F$ is a matching of fixed size, then $\ex_F(G)$ is computable in polynomial time.
This last result can be considered an algorithmic version of the celebrated Erd\H{o}s-Ko-Rado \nolinebreak Theorem.

The proof of [ASS'09] relied on a variety of tools from extremal graph theory, one of them
being Tur\'an's theorem. One of the main challenges we have to overcome in order to prove our hypergraph extension is the lack of a Tur\'an-type theorem for $k$-graphs. To circumvent this, we develop a completely new graph theoretic approach for proving such hardness results.


\end{abstract}

\section{Introduction}
Graph modification problems are concerned with the algorithmic problem of computing the minimal number of operations (vertex deletions, edge deletions, edge additions, etc.) needed to turn an input graph into a graph having a certain given property. The study of such problems has a long and rich history, going back to the work of Yannakakis \cite{Yannakakis0,Yannakakis1}
in the late 70's, and receiving significant attention since then, as evidenced by a recent survey \cite{CDFG} referencing hundreds of papers. See also \cite{AH,BBD,Cai,G,LY,Mancini,MS,NSS} for some notable results.

Here we focus on edge-modification problems for graphs and hypergraphs. Let us introduce some notation. A {\em $k$-uniform hypergraph} ({\em $k$-graph} for short) consists of a set of vertices $V$ and a set of edges $E$, such that every $e \in E$ is a subset of $V$ of size $k$. Thus, a 2-uniform hypergraph is just a graph.
For a family of $k$-graphs $\mathcal{F}$, a $k$-graph $G$ is {\em $\mathcal{F}$-free} if it contains no copy of any $F \in \mathcal{F}$.  
We use $\ex_{\mathcal{F}}(G)$ to denote the largest number of edges in an $\mathcal{F}$-free subgraph of $G$. Also, let $\rem_{\mathcal{F}}(G)$ denote the minimum number of edges whose deletion turns $G$ into an $\mathcal{F}$-free graph. Clearly, we have
\begin{equation}\label{eqremex}
\rem_{\mathcal{F}}(G) = e(G)-\ex_{\mathcal{F}}(G)\;.
\end{equation}
When $\mathcal{F} = \{F\}$, we write $\ex_{F}(G)$ and $\rem_F(G)$. 

The $\mathcal{F}$-freeness edge-modification problem is the algorithmic problem of computing $\rem_{\mathcal{F}}(G)$ for an input $k$-graph $G$. Equivalently (via \eqref{eqremex}), it is the algorithmic problem of computing $\ex_{\mathcal{F}}(G)$.
Note that this can be thought of as an algorithmic version of the Tur\'an problem, which is the central problem of extremal graph theory. Indeed, the Tur\'an problem asks to determine $\ex(n,\mathcal{F}) := \ex_{\mathcal{F}}(K_n)$, where $K_n$ is the complete graph on $n$ vertices.

Yannakakis asked in the early 80's \cite{Yannakakis0,Yannakakis1} if it is possible to prove a general NP-hardness result, showing
that computing $\rem_{\mathcal{F}}(\cdot)$ is NP-hard for a large collection of graph-families $\mathcal{F}$. A seminal result in the area of edge-modification problems is the theorem of Alon, Shapira and Sudakov \cite{Alon_Shapira_Sudakov}, which resolved Yannakakis's problem by proving the following:


\begin{theorem}[\cite{Alon_Shapira_Sudakov}]\label{thm:Alon_Shapira_Sudakov}
	For every graph-family $\mathcal{F}$ consisting of non-bipartite graphs, it is NP-hard to compute $\rem_{\mathcal{F}}(\cdot)$. Moreover, for every fixed $\delta > 0$, it is NP-hard to approximate $\rem_{\mathcal{F}}(G)$ to within an additive error of $N^{2-\delta}$ on input graphs $G$ with $N$ vertices.
\end{theorem}


The proof of Theorem \ref{thm:Alon_Shapira_Sudakov} is quite involved. It relies, among other things, on a highly non-trivial result in extremal graph theory, also proved in \cite{Alon_Shapira_Sudakov}. This result states that for every $r \geq 2$ and for every graph $F$ of chromatic number $r+1$, there is $c = c(F) > 0$ such that if $G$ is an $n$-vertex graph with minimum degree at least $(1-c)n$,
then the distance of $G$ to being $F$-free is close to the distance of $G$ to being $r$-colorable\footnote{More precisely, these two distances differ by at most $n^{2-c}$.}. We note that this statement is stronger than Tur\'an's theorem, which determines $\ex_{K_r}(K_n)$.
This result is then combined with the fact that estimating the distance to $r$-colorability is NP-hard (as well as with several additional ingredients) to establish Theorem \ref{thm:Alon_Shapira_Sudakov}.

Ailon and Alon \cite{AA} conjectured that Theorem \ref{thm:Alon_Shapira_Sudakov} can be extended to $k$-uniform hypergraphs. Our main result is a proof of this conjecture for the case of finite $k$-graph families $\mathcal{F}$. To state this result, let us recall the following definition: A $k$-graph $F$ is {\em $k$-partite} if there is a partition $V(F) = V_1 \cup \dots \cup V_k$ such that every edge of $F$ intersects each of the parts $V_1,\dots,V_k$. This is the natural extension of the notion of bipartiteness to $k$-graphs. Now, our main result is as follows:


\begin{theorem}\label{thm:main}
	Let $\mathcal{F}$ be a finite family of non-$k$-partite $k$-graphs. For every $\delta > 0$, it is NP-hard to approximate $\rem_{\mathcal{F}}(G)$ up to an additive error of $N^{k-\delta}$ for $k$-uniform inputs $G$ with $N$ vertices.
\end{theorem}

Let us explain the main challenge we had to overcome in order to prove Theorem \ref{thm:main}. A natural attempt to extend Theorem \ref{thm:Alon_Shapira_Sudakov} to $k$-graphs would be to first extend the extremal graph-theoretic results we mentioned in the paragraph following the statement of Theorem \ref{thm:Alon_Shapira_Sudakov}. However, any result of this sort is far out of reach of the current state of knowledge on extremal numbers of hypergraphs. Indeed, note that the aforementioned result from \cite{Alon_Shapira_Sudakov} (when applied to $G = K_n$) shows that the $n$-vertex $F$-free graphs with the maximum number of edges are close to being $r$-colorable, recovering the classical Erd\H{o}s-Stone theorem \cite{ErdosStone}. However, for $k$-graphs with $k \geq 3$, the structure of the extremal $F$-free hypergraphs is not known even for very simple hypergraphs, such as $K^{(3)}_4$ ($K^{(k)}_t$ denotes the complete $k$-graph on $t$ vertices). In fact, even the asymptotic value of 
$\ex(n,K^{(3)}_4)$
is not known; see \cite{Keevash_survey} for an overview of the current state of knowledge in this area. Thus, to prove Theorem \ref{thm:main}, one has to take a completely different approach which avoids the determination
of quantities such as $\ex(n,K^{(3)}_t)$. In particular, the approach of \cite{Alon_Shapira_Sudakov}, which criticially relies on ``knowing'' these quantities, cannot be used. In this paper we indeed develop such an approach, which turns out to be much simpler
than the one used in \cite{Alon_Shapira_Sudakov} for the proof of Theorem \ref{thm:Alon_Shapira_Sudakov}. In particular, we obtain
a much simpler and shorter proof of Theorem \ref{thm:Alon_Shapira_Sudakov} for the case of finite families $\mathcal{F}$.

Given Theorem \ref{thm:Alon_Shapira_Sudakov}, it is natural to look for a characterization of the graphs $F$ for which
computing $\rem_F(G)$ is NP-hard. And indeed, one of the open problems raised in \cite{Alon_Shapira_Sudakov}
is to determine for which bipartite $F$, computing $\rem_F(G)$ is NP-hard. They further noted that computing
this function is polynomial-time solvable when $F$ is a star.
In \cite{GLS_trees}, the authors conjectured that computing $\rem_F(G)$ is polynomial-time solvable if $F$ is a star forest and is NP-hard otherwise. They proved this conjecture in the case that $F$ is a forest. Here we investigate the analogous problem for hypergraphs, making the following conjecture, stating that for $k \geq 3$, the class of $F$'s for which
computing $\rem_F(G)$ is polynomially solvable is much more restrictive than in the case $k=2$. Recall that a hypergraph is a {\em matching} if it consists of pairwise vertex-disjoint edges.

\begin{conjecture}\label{conjecture}
For every $k\geq 3$ and every $k$-graph $F$, computing $\rem_F(G)$ is polynomial-time solvable if $F$ is a matching and NP-hard otherwise.
\end{conjecture}
\noindent
As our second result, we prove the positive part of the conjecture.


\begin{theorem}\label{thm:matching}
For every $k \geq 2$ and every $k$-uniform matching $M$, $\rem_{M}(\cdot)$ can be computed in polynomial time.
\end{theorem}

Recall that the celebrated Erd\H{o}s-Ko-Rado theorem \cite{EKR} states that for every $n \geq 2k$, an
$n$-vertex $k$-graph in which every pair of edges intersect can have at most ${n-1 \choose k-1}$ edges. Observe that this is equivalent
to the statement that for the $k$-graph $M_2$ consisting of $2$ disjoint edges (i.e. the matching of size 2), we have $\ex_{M_2}(K^{(k)}_n)={n-1 \choose k-1}$. Hence, Theorem \ref{thm:matching}, which gives an efficient algorithm for computing $\ex_{M_2}(G)$ for every $G$, can be considered an algorithmic version of the Erd\H{o}s-Ko-Rado \nolinebreak theorem.

For $k\geq 3$ and $1 \leq t \leq k-1$ let $F_t$ denote the $k$-graph consisting of $2$ edges sharing $t$ vertices.
Note that an equivalent way to state Conjecture \ref{conjecture} is that if $F$ contains a copy of $F_t$ for some $1 \leq t \leq k-1$, then
computing $\rem_{F_t}(G)$ is NP-hard. As a special case of Conjecture \ref{conjecture}, we also show that for these
``minimally'' hard $k$-graphs, computing $\rem_{F}(G)$ is indeed NP-hard.

\begin{proposition}\label{prop1}
For every $F_t$, computing $\rem_{F_t}(G)$ is NP-hard.
\end{proposition}

\subsection{Overview of proofs}

Let us give some details of the proof of Theorem \ref{thm:main}. 
For simiplicity, let us consider $\rem_F(G)$ for a given $k$-graph $F$. 
Our first observation is that it is more convenient to consider a {\em partite} version of this problem; namely, we only consider input graphs $G$ which are homomorphic to $F$. In other words, the input $G$ has a vertex partition with parts corresponding to the vertices of $F$, and edges of $G$ can only exist between $k$-tuples of parts corresponding to an edge of $F$. This partition will be given as part of the input. 
If $F$ is a {\em core}, meaning that every homomorphism from $F$ to itself is an isomorphism (we will define this precisely in Section \ref{sec:1}), then every copy of $F$ in such a $k$-graph $G$ must respect the vertex partition. This makes the copies of $F$ easier to handle.

To establish Theorem \ref{thm:main} for a given $k$-graph $F$, we consider the core $L$ of $F$, prove hardness for the problem of computing $\rem_L(\cdot)$ (exactly), and then show that computing $\rem_L(\cdot)$ reduces to approximating $F$ up to an error of $N^{k-\delta}$ on $N$-vertex inputs.

Thus, the main task is to prove that $\rem_F(\cdot)$ is hard when $F$ is a core. The key new insight in our proof is that it is now enough to reduce $\rem_L(\cdot)$ to $\rem_F(\cdot)$ for $L$ belonging to two special families of hypergraphs: (i) graph (i.e., 2-uniform) cycles; and (ii) complete hypergraphs. We note that this reduction is in fact for the partite versions of these problems (and this is crucial for the reduction to work). Roughly speaking, we show (as a special case of Lemma \ref{lem:main}) that if the projection of $E(F)$ to a subset of vertices of $F$ forms an induced (2-uniform) cycle of length at least 4 or a complete $(s-1)$-uniform hypergraph on $s$ vertices (for some $3 \leq s \leq k$), then computing $\rem_F(\cdot)$ is at least as hard as computing $\rem_L(\cdot)$ for $L$ being the corresponding cycle or complete $(s-1)$-uniform $s$-vertex hypergraph. We then show (see Lemma \ref{lem:core hardness}) that every non-$k$-partite $k$-graph $F$ admits one of these structures. Finally, we show that $\rem_L(\cdot)$ is hard for cycles and complete $(s-1)$-uniform $s$-vertex hypergraphs (see Lemmas \ref{lem:simplex} and \ref{lem:cycle}).

Let us now comment on the proof of Theorem \ref{thm:matching}.
Let $G$ be a $k$-uniform hypergraph, and let $M_r$ denote the $k$-uniform matching of size $r$.
Let $\mathcal{M}_r(G)$ be the set of all inclusion-wise maximal subgraphs of $G$ with no matching of size $r$; namely, every $F \in \mathcal{M}_r(G)$ is $M_r$-free, and every subgraph $H$ of $G$ which is $M_r$-free is a subgraph of some $F \in \mathcal{M}_r(G)$. It is easy to see that if we can generate $\mathcal{M}_r(G)$ in polynomial time, then we can design a polynomial-time algorithm for computing $\ex_{M_r}(G)$, as required by Theorem \ref{thm:matching}. The main idea behind the proof of Theorem \ref{thm:matching} is to devise an inductive procedure for generating $\mathcal{M}_r(G)$ (actually a superset thereof which is also of polynomial size). As is sometimes the case with such arguments, the main difficulty is coming up with the ``correct'' inductive assumption, which enables one to carry out the induction process.

Let us finally remark about a potential connection between our investigation here and the 
celebrated Dichotomy Theorem of Hell and Ne\v{s}et\v{r}il \cite{HN} (which is a special case of the even more famous Dichotomy Theorem
of Bulatov and Zhuk \cite{Bulatov,Zhuk} on CSPs). The result of \cite{HN} is about the complexity of deciding if an input graph $G$ has a homomorphism into a given fixed graph $H$, stating that this problem is polynomial-time solvable if $H$ is bipartite and NP-hard otherwise. It is easy to see that
this problem is a CSP, in which we need to assign colors to vertices and have a constraint for each edge of $G$.
Note that $\rem_F(\cdot)$ can also be thought of as a CSP, where we need to decide for each edge of the
$k$-graph if we want to remove it or not, and where we have a constraint $C$ for each copy of $F$ in the input $G$, so that in order to
satisfy $C$ we have to remove at least one of the edges of the copy of $F$ for which $C$ was created. As a CSP this problem is trivial
since we can always set all variables to $1$, i.e. we can always turn $G$ into an $F$-free $k$-graph by removing all edges.
The problem is of course to satisfy all constraint $C$, while setting to $1$ as few variables as possible. It is thus interesting to note that although $\rem_F(\cdot)$
is not a CSP per-se, the ideas we use in order to tackle Conjecture \ref{conjecture} are very much in the spirit of those
used in \cite{HN}. In particular, in both cases one has to find the ``minimal'' graphs/$k$-graphs for which the problem
is hard, and in both cases the notion of a {\em core} plays a crucial role. These ideas were also used in \cite{Bulatov,Zhuk}. It would thus
be interesting to further explore this possible connection between these two problems.

\subsection{Paper organization}\label{subsec:overview}
The proof of Theorem \ref{thm:main} has two main steps, one of which is graph-theoretic in nature, and one
which is more complexity-theoretic. We give the first part in Section \ref{sec:1}, and the second part in Section \ref{sec:triangle reduction}.
Theorem \ref{thm:matching} is proved in Section \ref{sec:matching}. As the proof of Proposition \ref{prop1} uses more routine arguments, we give it in the appendix.


\section{Proof of Theorem \ref{thm:main}}\label{sec:1}

We begin by recalling some basic facts regarding $k$-graph homomorphism.
A homomorphism $\phi$ from a $k$-graph $G$ to a $k$-graph $F$ is a mapping $\phi:V(G)\mapsto V(F)$ which maps edges of $G$ to edges of $F$. If such a mapping exists then we say that $G$ is {\em $F$-partite}.
Observe that a $k$-graph $G$ is $k$-partite if and only if $G$ is $F$-partite for $F$ being a single $k$-edge.
Suppose from now on that $F$ is a $k$-graph with vertex-set $[f]=\{1,\ldots,f\}$.
Note that every $F$-partite $k$-graph $G$ has a partition\footnote{This partition is of course obtained by setting $V_j=\phi^{-1}(i)$ for every $1 \leq i \leq f$.} $V(G) = V_1 \cup \dots \cup V_f$ such that every edge $e$ of $G$ contains at most one vertex from each of the sets $V_1,\dots,V_f$, and furthermore, if an edge of $G$ contains vertices from the sets $V_{i_1},\ldots,V_{i_k}$, then $\{i_1,\ldots,i_k\} \in E(F)$. We call such a partition $V_1,\dots,V_f$ an {\em $F$-partition} of $G$. A copy of $F$ in $G$ is called {\em canonical} if it is of the form $v_1,\dots,v_f$ with $v_i \in V_i$ playing the role of $i$ for every $i = 1,\dots,f$.

For an $F$-partite $k$-graph $G$ with $F$-partition $\mathcal{P} = (V_1,\dots,V_f)$, let $\rem_k(G,\P,F)$ denote the minimum number of edges of $G$ that need to be deleted to destroy all canonical copies of $F$.
We stress that in this problem, $\mathcal{P}$ is given as part of the input.

In Section \ref{subsec:special} we prove three key lemmas which establish the NP-hardness of computing the function $\rem_k(G,\P,F)$ for certain special cases.
In Section \ref{subsec:combine} we will show how these lemmas can be combined in order to handle all finite families of non-$k$-partite $k$-graphs, thus proving Theorem \nolinebreak \ref{thm:main}.

\subsection{The Three Key Lemmas}\label{subsec:special}

Our goal in this section is to prove Lemmas \ref{lem:simplex}, \ref{lem:cycle} and \ref{lem:main}.
The first two deal with special cases of Theorem \ref{thm:main}, while the third
will help us in combining the first two in order to deal with all finite families of non-$k$-partite $k$-graphs.
The first special case of Theorem \ref{thm:main} we address is when $F$ is the complete $k$-graph on $k+1$ vertices, which we denote by $K_{k+1}^{(k)}$.

\begin{lemma}\label{lem:simplex}
For every $k \geq 2$, computing $\rem_k(G,\P,K_{k+1}^{(k)})$ is NP-hard.
\end{lemma}

As it turns out (see also the discussion after the proof of Lemma \ref{lem:simplex}), to prove this lemma one can first prove the graph case (i.e., the case $k=2$), and then prove the cases $k \geq 3$ via a reduction from the case $k=2$.
Note that for $k=2$, $K_{3}^{(2)}$ is the $3$-cycle, denoted $C_3$.
Since the proof of this special case is less graph-theoretic in nature, we defer the proof of the following lemma to Section \ref{sec:triangle reduction}.

\begin{lemma}\label{lem:triangle}
Computing $\rem_2(G,\P,C_3)$ is NP-hard.
\end{lemma}

\noindent
We now prove Lemma \ref{lem:simplex} by reducing it to Lemma \ref{lem:triangle}.

\begin{proof}[Proof of Lemma \ref{lem:simplex}]
The case $k=2$ is the statement of Lemma \ref{lem:triangle}. Fixing $k \geq 3$, we now reduce the problem of computing $\rem_2(G,\P,C_3)$ to the problem of computing $\rem_k(H,\mathcal{Q},K_{k+1}^{(k)})$. Let $G$ be a 3-partite graph with $C_3$-partition $\P = (A,B,C)$. Define a $k$-graph $H$ as follows: The vertex-set of $H$ consists of $V(G)$ and new vertices $v_1,\dots,v_{k-2}$. For every edge $e \in E(G)$, add to $H$ the edge $e \cup \{v_1,\dots,v_{k-2}\}$. Finally, for each triple $(a,b,c) \in A \times B \times C$ and for each $1 \leq i \leq k-2$, add to $H$ the edge $\{a,b,c\} \cup \{v_j : j \in [k-2] \setminus \{i\}\}$.
Clearly, $H$ is $K_{k+1}^{(k)}$-partite with partition $\mathcal{Q} = (A,B,C,\{v_1\},\dots,\{v_{k-2}\})$, hence every copy of $K^{(k)}_{k+1}$ has vertices $a,b,c, v_1,\ldots,v_{k-2}$ for some $a \in A, b \in B, c\in C$.
Moreover, if such a $(k+1)$-tuple $(a,b,c,v_1,\ldots,v_{k-2})$ forms a copy of $K^{(k)}_{k+1}$ then $a,b,c$ is a triangle in $G$.
		We claim that
		\begin{equation}\label{eq:simplex}
			\rem_2(G,(A,B,C),K_3) = \rem_k(H,(A,B,C,v_1,\dots,v_{k-2}),K_{k+1}^{(k)}).
		\end{equation}
		First, let $E \subseteq E(G)$ be an edge-set such that $G - E$ has no triangles. Let $E' \subseteq E(H)$ be the set of all edges $e \cup \{v_1,\dots,v_{k-2}\}$, $e \in E$. By the observation we made prior to \eqref{eq:simplex}, we see that $H - E'$ is $K_{k+1}^{(k)}$-free, showing that the RHS of \eqref{eq:simplex} is not larger than the LHS.
		
		In the other direction, let $E' \subseteq E(H)$ be an edge-set of minimal size such that $H - E'$ is $K_{k+1}^{(k)}$-free.
		If, for some $(a,b,c) \in A \times B \times C$, there is an edge $e' \in E'$ which contains $a,b,c$,
		then we can remove $e'$ from $E'$, and instead add to $E'$ the edge $e'' = \{a,b\} \cup \{v_1,\dots,v_{k-2}\}$. Note that $e'$ is part of at most one copy of $K_{k+1}^{(k)}$ in $H$, namely the copy $a,b,c,v_1,\dots,v_{k-2}$ (which is a copy of $K_{k+1}^{(k)}$ if and only if $a,b,c$ is a triangle in $G$), and $e''$ is also an edge of this copy. Hence, after replacing $e'$ with $e''$ we still have that $H - E'$ is $K_{k+1}^{(k)}$-free. After repeating this operation for each triangle in $G$, we may assume that every edge in $E'$ is of the form $e \cup \{v_1,\dots,v_{k-2}\}$ for some $e \in E(G)$. Now, let $E$ be the set of all $e \in E(G)$ such that $e \cup \{v_1,\dots,v_{k-2}\} \in E'$. Then $G-E$ is $K_3$-free, as otherwise $H-E'$ would not be $K_{k+1}^{(k)}$-free. This shows that $|E'| \geq \rem_2(G,(A,B,C),K_3)$, completing \nolinebreak the \nolinebreak proof.
\end{proof}

Examining the proof above suggests the following approach to proving Lemma \ref{lem:triangle}: Take an instance graph $G$ to the Vertex-Cover problem which is triangle free, add a new vertex $v$, and connect $v$ to all of $G$'s vertices. However, in order for the new graph to be $3$-partite, we need $G$ to be bipartite. And in this case the Vertex Cover problem is actually solvable in polynomial time, thanks to the K\"{o}nig-Egerv\'ary theorem. This perhaps explains the hardness of proving \nolinebreak Lemma \nolinebreak \ref{lem:triangle}.

Our second key lemma concerns the case where $k=2$ and $F$ is an $\ell$-cycle, which is denoted by \nolinebreak $C_{\ell}$.

\begin{lemma}\label{lem:cycle}
For every $\ell \geq 3$, computing $\rem_2(G,\P,C_{\ell})$ is NP-hard.
\end{lemma}

\begin{proof}
		The case $\ell = 3$ is the statement of Lemma \ref{lem:triangle}.
		Let now $\ell \geq 4$. We reduce the problem of computing $\rem_2(G,\P,C_{3})$ to the problem of computing $\rem_2(G',\P',C_{\ell})$.
		So let $G$ be an $n$-vertex input graph with a $C_3$-partition $\P = (A,B,C)$. Construct a new graph $G'$ as follows: $G'$ has $\ell$ parts $A,V_1,\dots,V_{\ell-3},B,C$.
		We set $G'[B,C] = G[B,C]$ and $G'[A,C] = G[A,C]$ (i.e., the bipartite graphs $(B,C)$ and $(A,C)$ are the same in $G$ and $G'$). Next, for each edge $e = ab \in E(G)$ with $a \in A, b\in B$, add new vertices $v^e_1 \in V_1,\dots,v^e_{\ell-3} \in V_{\ell-3}$ and add the path $P_e := (a,v^e_1,\dots,v^e_{\ell-3},b)$ to $G'$. This completes the definition of $G'$. It is immediate that $G'$ is $C_{\ell}$-partite with partition $(A,V_1,\dots,V_{\ell-3},B,C)$.
		Also, the number of vertices of $G'$ is $n + (\ell-3) \cdot e_G(A,B) \leq n + (\ell-3)n^2 = \poly(n)$. Finally, it is easy to see that every canonical $C_{\ell}$-copy in $G'$ is of the form $(a,P_{ab},b,c,a)$ for a triangle $a,b,c$ in $G$.
		We claim that
		\begin{equation}\label{eq:cycle}
		\rem_2(G,(A,B,C),C_3) = \rem_2(G',(A,V_1,\dots,V_{\ell-3},B,C),C_{\ell}).
		\end{equation}
		First, let $E \subseteq E(G)$ be a set of edges such that $G-E$ is $C_3$-free. Define an edge-set $E' \subseteq E(G')$ as follows. For each $e \in E$, if $e \in E(B,C)$ or $e \in E(A,C)$, then put $e$ into $E'$. And if $e \in E(A,B)$ then put one of the edges of the path $P_e$ into $E'$. Then $|E'| = |E|$, and $G' - E'$ has no canonical\footnote{$G'-E'$ might contain non-canonical copies of $C_{\ell}$. For example, if $\ell$ is even, then $G'-E'$ might contain a copy of $C_{\ell}$ between (say) $B$ and $C$.} $C_{\ell}$-copies. This shows that the RHS of \eqref{eq:cycle} is not larger than the LHS.
		
		In the other direction, let $E' \subseteq E(G')$ be a set of edges such that $G'-E'$ has no canonical $C_{\ell}$-copies. Define an edge-set $E \subseteq E(G)$ as follows. For each $e \in E'$, if $e \in E(B,C)$ or $e \in E(A,C)$, then put $e$ into $E$. And if $e$ is an edge of $P_{ab}$ for some $ab \in E(G)$ with $a \in A$, $b \in B$, then put $ab$ into $E$. Then $|E| \leq |E'|$, and $G - E$ is $C_3$-free. This completes the proof.
\end{proof}
\noindent
Lemma \ref{lem:main} below is our third key lemma, and forms the heart of the proof.
	
\begin{lemma}\label{lem:main}
Let $F$ be a $k$-graph with vertex-set $[f]$, and suppose that there are integers $3 \leq \ell \leq f$ and $2 \leq s \leq k$ such that the following holds:
\begin{enumerate}
			\item For every edge $a \in E(F)$, $|a \cap [\ell]| \leq s$.
			\item Computing $\rem_s(G,\P,L)$ is NP-hard, where $L$
			is the $s$-graph on the vertex-set $[\ell]$ with edge-set $\{e \cap [\ell] : e \in E(F), |e \cap [\ell]| = s\}$.
		\end{enumerate}
	Then computing $\rem_k(G,\P,F)$ is NP-hard.
\end{lemma}

\begin{proof}
		To prove the lemma, we reduce the problem of computing $\rem_s(G,\P,L)$ to the problem of computing $\rem_k(H,\mathcal{Q},F)$. Let $G$ be an $n$-vertex $L$-partite $s$-graph with $L$-partition $\P = (V_1,\dots,V_{\ell})$.
		We will use the letters $a,b$ to denote edges of $F$, and $e,e'$ to denote edges of $G$.
		We construct an $F$-partite $k$-graph $H$ with $F$-partition $\mathcal{Q} = (V_1,\dots,V_{\ell},V_{\ell+1},\dots,V_f)$, where $V_{\ell+1},\dots,V_f$ are new parts with $|V_i| = N := n^s$ for $\ell+1 \leq i \leq f$.
		The edges of $H$ are defined as follows: for each edge $a$ of $F$, if $|a \cap [\ell]| \leq s-1$, then put a complete $k$-partite $k$-graph on the parts $(V_i)_{i \in a}$. Let now $a \in E(F)$ with $|a \cap [\ell]| \geq s$. Then $|a \cap [\ell]| = s$ by Item 1 in the lemma, and $a \cap [\ell] \in E(L)$ by the definition of $L$. Now, for each choice of $v_i \in V_i$ for $i \in a$, make $\{v_i : i \in a\}$ an edge of $H$ if and only if $\{v_i : i \in a \cap [\ell]\}$ is an edge of $G$. Thus, if $\{v_i : i \in a \cap [\ell]\}$ is an edge of $G$, then $\{v_i : i \in a\}$ is an edge of $H$ for every $(k-s)$-tuple $(v_i)_{i \in a \setminus [\ell]}$, $v_i \in V_i$. Note that there are $N^{k-s}$ such choices of $(v_i)_{i \in a \setminus [\ell]}$.
		The resulting $k$-graph is $H$. By definition, $H$ is $F$-partite with $F$-partition $\mathcal{Q} = (V_1,\dots,V_f)$. Also, $|V(H)| = |V(G)| + (f-\ell) \cdot N = n + (f-\ell) \cdot n^s = \poly(n)$. We claim that
		\begin{equation}\label{eq:cycle reduction}
		\rem_k(H,(V_1,\dots,V_f),F) = \rem_s(G,(V_1,\dots,V_{\ell}),L) \cdot N^{k-s}.
		\end{equation}
		
		First, let $E \subseteq E(G)$ be such that $G-E$ has no canonical copies of $L$.
		For each edge $b \in E(L)$, fix an edge $a_b$ of $F$ with $a_b \cap [\ell] = b$ (such $a_b$ exists by the definition of $L$). 
		Now, define a set of edges $E' \subseteq E(H)$ as follows. For each $b \in E(L)$ and each edge $\{v_i : i \in b\} \in E$, add to $E'$ all edges of the form $\{v_i : i \in b\} \cup \{u_i : i \in a_b \setminus b\}$ with $u_i \in V_i$; namely, add to $E'$ all edges of $H$ which go between the parts $(V_i)_{i \in a_b}$ and contain $\{v_i : i \in b\}$.
		Then $|E'| = |E| \cdot N^{k-s}$.
		We claim that $H - E'$ has no canonical $F$-copies\footnote{Actually, $H - E'$ has no canonical $F'$-copies where $F'$ is the subgraph of $F$ consisting of the edges $\{a_b:b \in E(L)\}$.}. Indeed, suppose that $v_1,\dots,v_f$ make a canonical $F$-copy in $H$, with $v_i \in V_i$.
		For every $b \in E(L)$, we have $\{v_i : i \in a_b\} \in E(H)$ (as $v_1,\dots,v_f$ make a canonical copy of $F$). By the definition of $H$, it follows that $\{v_i : i \in b\}$ is an edge of $G$. Thus, $v_1,\dots,v_{\ell}$ make a canonical copy of $L$ in $G$. Therefore, there is $b \in E(L)$ such that $\{v_i : i \in b\} \in E$ (as $G-E$ has no canonical copies of $L$). Now, by the definition of $E'$, the edge $\{v_i : i \in a_b\}$ is in $E'$, meaning that $E'$ contains some edge of the copy $v_1,\dots,v_f$. This proves that $H - E'$ has no canonical $F$-copies, implying that the LHS in \eqref{eq:cycle reduction} is not larger than the RHS.
		
		In the other direction, let $E' \subseteq E(H)$ be a set of edges such that $H-E'$ has no canonical $F$-copies. For each $\ell+1 \leq i \leq f$, sample a vertex $v_i \in V_i$ uniformly at random. Let us introduce some notation.
		Let $E'_1$ be the set of edges $e' \in E'$ which intersect at most $s-1$ of the parts $V_1,\dots,V_{\ell}$, and let $E'_2$ be the set of edges $e' \in E'$ which intersect exactly $s$ of the parts $V_1,\dots,V_{\ell}$. For $i=1,2$, let $E''_i$ be the set of edges in $E'_i$ which are contained in $V_1 \cup \dots \cup V_{\ell} \cup \{v_{\ell+1},\dots,v_f\}$. Observe that
		$\mathbb{E}[|E''_1|] \leq |E'_1|/N^{k-s+1}$
		and
		$\mathbb{E}[|E''_2|] = |E'_2|/N^{k-s}.$
		Now consider any given outcome for $v_{\ell+1},\dots,v_f$ such that $E''_1 = \emptyset$.
		Let $E$ be the set of all edges of $G$ which are contained in some edge of $E''_2$.
		Note that $|E| \leq |E''_2|$, because every edge in $E''_2$ intersects $V(G) = V_1 \cup \dots \cup V_{\ell}$ in exactly $s$ vertices\footnote{It may be the case that $|E| < |E''_2|$, since an edge of $E$ might be contained in several edges of $E''_2$ (if $F$ has several edges which have the same size-$s$ restriction to $[\ell]$).}.
		We claim that $G-E$ has no canonical copy of $L$.
		To this end, we show that if there is a canonical copy of $L$ in $G-E$ on the vertices $v_1 \in V_1,\ldots ,v_{\ell} \in V_{\ell}$, then there is a canonical copy of $F$ in $H-E'$ on the vertices $v_1,\ldots,v_{\ell},v_{\ell+1},\dots,v_f$, which would be a contradiction to $H-E'$ being $F$-free.
		Indeed, take any edge $a \in E(F)$, and let $b = a \cap [\ell]$ be the restriction of $a$ to $[\ell]$.
		If $|b| < s$ then $\{v_i : i \in a\}$ is an edge of $H$ (by the definition of $H$), and since we assume that $E''_1=\emptyset$, we also have $\{v_i : i \in a\} \notin E'$.
		If $|b|=s$ then $b$ is an edge of $L$ (by the definition of $L$). Also, as $v_1,\dots,v_{\ell}$ are assumed to make a canonical copy of $L$ in $G-E$, we have that
		$\{v_i : i \in b\} \in E(G)$ and hence $\{v_i : i \in a\} \in E(H)$ (by the definition of $H$). Moreover, $\{v_i : i \in a\} \notin E'_2$ because
		$\{v_i : i \in b\} \notin E$. So it follows that $\{v_i : i \in a\} \in H-E'$. This proves our claim that $G-E$ has no canonical copies of $L$.
		
		We have thus proved that if $E''_1 = \emptyset$, then $|E''_2| \geq \rem_s(G,(V_1,\dots,V_{\ell}),L) =: M$. Note that $M \leq e(G) \leq \binom{n}{s} \leq N$. As $\mathbb{E}[|E''_1|] \leq |E'_1|/N^{k-s+1}$, we have $\mathbb{P}[E''_1 = \emptyset] \geq 1 - |E'_1|/N^{k-s+1}$. Thus,
		$$
		\mathbb{E}\left[|E''_2|\right] \geq \mathbb{E}\left[|E''_2|~|~E''_1 = \emptyset\right] \cdot \mathbb{P}\left[E''_1 = \emptyset\right] \geq M \cdot \left(1 - \frac{|E'_1|}{N^{k-s+1}} \right).
		$$
		On the other hand, $\mathbb{E}[|E''_2|] = |E'_2|/N^{k-s}$. It follows that
		$$
		|E'_2| \geq N^{k-s} \cdot M \cdot \left(1 - \frac{|E'_1|}{N^{k-s+1}} \right) = M \cdot N^{k-s} - \frac{M |E'_1|}{N} \geq M \cdot N^{k-s} - |E'_1|.
		$$
		Thus, $|E'| = |E'_1| + |E'_2| \geq M \cdot N^{k-s} = \rem_s(G,(V_1,\dots,V_{\ell}),L) \cdot N^{k-s}$. This shows that the LHS of \eqref{eq:cycle reduction} is not smaller than the RHS, completing the proof of \eqref{eq:cycle reduction}, and hence proving the lemma.
	\end{proof}

\subsection{Putting Everything Together}\label{subsec:combine}

We first combine Lemmas \ref{lem:simplex}, \ref{lem:cycle} and \ref{lem:main} in order to prove a partite version of Theorem \ref{thm:main} for a single forbidden subgraph.
For the proof, we need to recall the notion of a shadow.
For a $k$-graph $F$ and an integer $2 \leq s \leq k$, the {\em $s$-shadow} of $F$, denoted $\partial_s(F)$, is the $s$-graph with vertex-set $V(F)$, where $a \in \binom{V(F)}{s}$ is an edge if and only if there is $e \in E(F)$ with $a \subseteq e$. Note that for $s = k$ we have $\partial_k(F) = F$.

\begin{lemma}\label{lem:core hardness}
Let $F$ be a non-$k$-partite $k$-graph. Then computing $\rem_k(G,\P,F)$ is NP-hard.
\end{lemma}

\begin{proof}
Write $V(F) = [f]$. Consider the graph $\partial_2(F)$. First, suppose that $\partial_2(F)$ contains an induced cycle of length at least $4$. Without loss of generality, let us suppose that this cycle is $(1,\dots,\ell,1)$ for some $\ell \geq 4$. Note that for every $e \in E(F)$ we have $|e \cap [\ell]| \leq 2$ because the cycle $(1,\dots,\ell,1)$ is induced. So the edge-set of this cycle is exactly $\{e \cap [\ell] : e \in E(F), |e \cap [\ell]| = 2\}$.
Moreover, by Lemma \ref{lem:cycle}, computing $\rem_2(G,\P,C_{\ell})$ is NP-hard. Thus, by Lemma \ref{lem:main} with $s=2$ and $L = C_{\ell}$, we get that computing $\rem_k(G,\P,F)$ is NP-hard.
	
From now on, we may assume that $\partial_2(F)$ has no induced cycle of length at least $4$; namely, $\partial_2 F$ is a chordal graph. It is well-known that chordal graphs are perfect. Also, $\chi(\partial_2(F)) \geq k+1$, because otherwise $F$ would be $k$-partite (indeed, a proper $k$-coloring of $\partial_2(F)$ corresponds to a partition $V(F) = V_1 \cup \dots \cup V_k$ such that no two vertices which are contained together in an edge of $F$ are in the same part, meaning that each edge intersects all of the parts). It follows that $\partial_2(F)$ contains a clique $X$ of size $k+1$. Let $Y \subseteq X$ be a minimum set with the property that $Y$ is not contained in any edge of $e$. Note that $Y$ is well-defined because $X$ is not contained in any edge (as edges have size $k$, whereas $|X| = k+1$). Also,
$|Y| \geq 3$ because every subset of $X$ of size $2$ is contained in an edge of $F$, as $X$ is a clique in $\partial_2(F)$.
Now, put $\ell := |Y|$ and $s := \ell-1 \leq k$, and let us assume, without loss of generality, that $Y = [\ell]$.
Then $|e \cap [\ell]| \leq s$ for every $e \in E(F)$ (since $Y$ is not contained in any edge of $F$) and $L := \{e \cap [\ell] : e \in E(F), |e \cap [\ell]| = s\}$ consists of all subsets of $Y$ of size $s$ (because every proper subset of $Y$ is contained in some edge, by the minimality of $Y$). Thus, $L$ is isomorphic to
$K_{s+1}^{(s)}$.
By Lemma \ref{lem:simplex}, computing $\rem_s(G,\P,K_{s+1}^{(s)})$ is NP-hard. Thus, by Lemma \ref{lem:main} with $L = K_{s+1}^{(s)}$, we get that computing $\rem_k(G,\P,F)$ is NP-hard. This completes the proof.
\end{proof}
\noindent

We are almost ready to prove Theorem \ref{thm:main}. The last thing we need is to introduce the notion of a core and some of its properties.
We say that a $k$-graph $L$ is a {\em core} of a $k$-graph $F$ if: $(i)$ $F$ is $L$-partite\footnote{In other words, there is a homomorphism from $F$ to $L$.}; $(ii)$ $L$ is a subgraph of $F$; and $(iii)$ $L$ has the smallest number of vertices among all $k$-graphs satisfying properties $(i)$ and $(ii)$. It is easy to see that all cores of a $k$-graph are isomorphic\footnote{Indeed, suppose that $L_1,L_2$ are both cores of $F$. Then $L_1$ is homomorphic to $L_2$ (by taking a homomorphism from $F$ to $L_2$ and restrincting it to $V(L_1)$) and similarly $L_2$ is homomorphic to $L_1$. Also, by the minimality of a core, both homomorphisms $\varphi : L_1 \rightarrow L_2$ and $\psi : L_2 \rightarrow L_1$ must be surjective. Indeed, if e.g.~$\varphi$ were not surjective, then by composing $\varphi$ with a homomorphism from $F$ to $L_1$, we would get a homomorphism from $F$ to a proper subgraph of $L_2$, a contradiction. So $|V(L_1)| = |V(L_2)|$ and $\varphi,\psi$ are in fact bijections. It follows that $L_1,L_2$ are isomorphic.}. Hence, we can talk about {\em the} core of $F$, which we denote by $\text{core}(F)$. Finally, we call a $k$-graph $L$ a core if it is the core of itself.
It is easy to see that if $L$ is a core then every homomorphism from $L$ to itself must in fact be an isomorphism\footnote{Indeed, the fact that $L$ is its own core implies that any homomorphism from $L$ to itself is bijective, which in turn implies that it is an isomorphism.}. We will need the following four simple facts regarding cores. We refer to the book \cite{HN_book} for a more thourough overview of these concepts.

\begin{lemma}\label{lem:core}
If $L$ is a core then for every $L$-partite graph $G$, every copy of $L$ in $G$ is canonical.
\end{lemma}

\begin{proof}
Since $G$ is $L$-partite, there is a homomorphism $\phi: G \rightarrow L$. This homomorphism is also a homomorphism from
every copy of $L$ in $G$ to $L$. As every homomorphism from a core to itself
must be an isomorphism, we infer that every copy of $L$ in $G$ must be canonical.
\end{proof}

\begin{lemma}\label{lem:core blowup}
If $L$ is a core, and $G$ is $L$-partite and $L$-free, then every $G$-partite graph is also $L$-free.
\end{lemma}
\begin{proof}
If $G'$ is $G$-partite, then $G$ is also $L$-partite. Indeed, if $\phi'$ is a homomorphism from $G'$ to $G$ and
$\phi$ is a homomorphism from $G$ to $L$, then the composed mapping $\phi \circ \phi'$ is a homomorphism from $G'$ to $L$.
Now suppose by contradiction that $G'$ has a copy $X$ of $L$, and consider the restriction of $\phi \circ \phi'$ to the vertex-set of $X$.
As we noted above, since $L$ is a core, such a homomorphism from $L$ to itself must be an isomorphism. In particular, $(\phi \circ \phi')|_X$ must be injective, and so
$\phi'|_X$ is injective. But then $\phi'$ maps the copy $X$ of $L$ in $G'$ to a copy of $L$ in $G$ (as $\phi'$ is a homomorphism),
contradicting the assumption that $G$ is $L$-free.
\end{proof}

\noindent
We write $G \rightarrow F$ to mean that there is a homomorphism from $G$ to $F$. 
\begin{lemma}\label{lem:hom_equivalence}
	Let $F_1,F_2$ be $k$-graphs and let $L_i$ be the core of $F_i$. Then the following are equivalent:
	\begin{itemize}
		\item $F_1 \rightarrow F_2$ and $F_2 \rightarrow F_1$.
		\item $L_1 \cong L_2$. 
	\end{itemize} 
\end{lemma}
\begin{proof}
	If $F_1 \rightarrow F_2$ and $F_2 \rightarrow F_1$ then $L_1 \rightarrow L_2$ and $L_2 \rightarrow L_1$ (as $L_i$ is a subgraph of $F_i$ and $F_i \rightarrow L_i$). Let $\phi : L_1 \rightarrow L_2$ and $\psi : L_2 \rightarrow L_1$ be homomorphisms. Then $\psi \circ \phi$ is a homomorphism from $L_1$ to itself and $\phi \circ \psi$ is a homomorphism from $L_2$ to itself. As $L_1,L_2$ are cores, $\psi \circ \phi$ and $\phi\circ \psi$ must be bijective, which in turn implies that $\phi,\psi$ are bijective and hence isomorphisms. 
	
	Conversely, suppose that $L_1 \cong L_2 =: L$. Then $F_1,F_2 \rightarrow L$ and $L$ is a subgraph of $F_1,F_2$, hence $F_1 \rightarrow F_2$ and $F_2 \rightarrow F_1$.
\end{proof}


\begin{lemma}\label{corehomomor}
For every $k$-graph $F$, there are $C=C(F)$ and $\varepsilon=\varepsilon(F)>0$, so that if $L$ is the core\footnote{Actually, in this lemma
$L$ does not need to be the core of $F$. It is enough for $F$ to be $L$-partite.} of $F$,
and $\ell=|V(L)|$, then every $k$-graph on $n \geq C$ vertices containing at least $n^{\ell-\varepsilon}$ copies of $L$ contains
a copy of $F$.
\end{lemma}
\begin{proof}
Set $f=|V(F)|$, and suppose $G$ is an $n$-vertex $k$-graph on at least $C$ vertices with $m=n^{\ell-\varepsilon}$ copies of $L$ with $\varepsilon$ and $C$ to be chosen below. Consider a random partition of $V(G)$ into
$\ell$ vertex sets $V_1,\ldots,V_{\ell}$ where each vertex is assigned randomly and independently to one of the sets. Suppose
$(v_1,\ldots,v_{\ell})$ is an $\ell$-tuple of vertices in $G$ which spans a copy of $L$ where $v_i$ plays the role of vertex
$i$ in $L$. Then with probability $\ell^{-\ell}$, vertex $v_i$ is placed in $V_i$ for each $1 \leq i \leq \ell$.
Hence, the expected number of such $\ell$-tuples is $m/\ell^{\ell}$. Fix a partition $V_1,\ldots,V_{\ell}$ with at least this many such $\ell$-tuples, and define an $\ell$-graph on the same set of vertices, by putting an edge on $v_1 \in V_1,\ldots,v_{\ell} \in V_{\ell}$ if they span a copy of $L$ as above.
Then we have an $n$-vertex $\ell$-graph with at least $n^{\ell-\varepsilon}/\ell^{\ell}$ edges. It is well known \cite{Erdos} that there is $\varepsilon'=\varepsilon'(\ell,f)$ so that every $n$-vertex $\ell$-graph with at least $n^{\ell-\varepsilon'}$ edges has a copy of the complete $\ell$-partite $\ell$-graph with each vertex
part of size $f$. It is easy to see that since $f=|V(F)|$, this gives a copy of $F$ in $G$.
Hence, we can pick $\varepsilon=\varepsilon'/2$ and $C$ large enough so that $n^{\ell-\varepsilon'/2}/\ell^{\ell} \geq n^{\ell-\varepsilon'}$ for all $n \geq C$.
\end{proof}

In the proof of Theorem \ref{thm:main} we will use the following notion:
For a $k$-graph $G$, the {\em $b$-blowup} of $G$ is the graph $G'$ obtained by replacing every vertex $v \in V(G)$ with a set $X_v$ of $b$ vertices (such that the sets $(X_v)_{v \in V(G)}$ are pairwise-disjoint), and adding for every $e \in E(G)$ a complete $k$-partite $k$-graph between the sets $(X_v)_{v \in e}$. Observe that a $b$-blowup of $G$ is $G$-partite.

\begin{proof}[Proof of Theorem \ref{thm:main}]
Let $\mathcal{F}$ be a finite family of non-$k$-partite $k$-graphs. As $\mathcal{F}$ is finite, there exists $F \in \mathcal{F}$ which is minimal in the homomorphism relation, namely, such that for every $F' \in \mathcal{F}$, if $F' \rightarrow F$ then $F \rightarrow F'$. Let $L$ be the core of $F$. By Lemma \ref{lem:hom_equivalence}, for every $F' \in \mathcal{F}$ it holds that either $\text{core}(F') = L$ or $F'$ is not homomorphic to $F$. 	
	
Put $\ell = |V(L)|$, and note that $L$ is also not $k$-partite. Fix $\delta > 0$.
We reduce the problem of computing $\rem_k(G,\P,L)$ exactly on $n$-vertex inputs to the problem of approximating $\ex_{\mathcal{F}}(G')$ up to an additive error of $N^{k-\delta}$ on $N$-vertex inputs $G'$, where $N=\poly(n)$ will be chosen later.
The former problem is NP-hard by Lemma \ref{lem:core hardness}, so this would prove Theorem \ref{thm:main}.
Let $G$ be an $n$-vertex $L$-partite $k$-graph with $L$-partition $\P$.
By Lemma \ref{lem:core}, every copy of $L$ in $G$ is canonical.
Hence, together with \eqref{eqremex}, we get
\begin{equation}\label{eqfinal1}
\rem_k(G,\P,L) = \rem_L(G) = e(G) - \ex_L(G)\;.
\end{equation}
By \eqref{eqfinal1}, computing $\rem_k(G,\P,L)$ is equivalent to computing $\ex_L(G)$.
We will now reduce the task of computing $\ex_L(G)$ to that of approximating $\ex_{\mathcal{F}}(G')$ up to an additive error of $N^{k-\delta}$ for $N$-vertex inputs $G'$. This would establish our reduction.

Since $F$ is $L$-partite (by definition of a core), Lemma \ref{corehomomor} gives us $C$ and $\varepsilon > 0$ such that for every $N\geq C$, every $N$-vertex $k$-graph containing $N^{\ell-\varepsilon}$ copies of $L$ must contain a copy of $F$.
	Let $G'$ be the $b$-blowup of $G$, where
	$$
	b = \max\left\{ 4n^{\frac{k}{\delta}}, 4n^{\frac{\ell}{\varepsilon}}, C \right\}.
	$$
	So $b = \poly(n)$.
	Let $X_1,\dots,X_n$ be the parts of the blowup corresponding to the vertices $1,\dots,n$ of $G$. So $(X_{i_1},\dots,X_{i_k})$ is a complete $k$-partite $k$-graph whenever $\{i_1,\dots,i_k\} \in E(G)$, and these are all the edges of $G'$.
	Also, $|X_i| = b$ for every $1 \leq i \leq n$.
	Put $N := |V(G')| = bn$, and note that $N = \poly(n)$ because $b = \poly(n)$. We claim that
	\begin{equation}\label{eq:approx}
	b^k \cdot \ex_L(G) \leq \ex_{\mathcal{F}}(G') \leq b^k \cdot \ex_L(G) + \frac{N^{\ell-\varepsilon}}{b^{\ell-k}}.
	\end{equation}
	For the left inequality, let $H$ be an $L$-free subgraph of $G$ with $\ex_L(G)$ edges. Let $H'$ be the $b$-blowup of $H$. Then $H'$ is a subgraph of $G'$ with $b^k \cdot \ex_L(G)$ edges. Since $H$ is $L$-partite (on account of being a subgraph of $G$ which is assumed to be $L$-partite),
	and since $H'$ is $H$-partite (on account of being a blowup of $H$), 
	we get that $H'$ is $L$-partite (by composing homomorphisms), and we infer from Lemma \ref{lem:core blowup} that
	$H'$ is $L$-free. We claim that $H'$ is $\mathcal{F}$-free. So fix any $F' \in \mathcal{F}$. If there is no homomorphism from $F'$ to $F$ then there is also no homomorphism from $F'$ to $L$ (as $L$ is a subgraph of $F$). As $H'$ is $L$-partite, it follows that $H'$ has no copies of $F'$. And if $F' \rightarrow F$, then by the choice of $F$ we have $\text{core}(F') = L$, so in particular $L$ is a subgraph of $F'$. As $H'$ is $L$-free, it is also $F'$-free. This proves that $H'$ is $\mathcal{F}$-free, as claimed. 
	It follows that $\ex_{\mathcal{F}}(G') \geq e(H') = b^k \cdot \ex_L(G)$, proving the left inequality in~\eqref{eq:approx}.
	
	To prove the right inequality in \eqref{eq:approx}, let $H'$ be a subgraph of $G'$ with $e(H') \geq b^k \cdot \ex_L(G) + \frac{N^{\ell-\varepsilon}}{b^{\ell-k}}$. Our goal is to show that $H'$ is not $\mathcal{F}$-free. To this end, it suffices to show that $H'$ contains a copy of $F$. By the choice of $C$ and $\varepsilon$ via Lemma \ref{corehomomor}, it suffices to show that $H'$ has at least $N^{\ell-\varepsilon}$ copies of $L$. To this end, for each $1 \leq i \leq n$, let $x_i \in X_i$ be a vertex chosen uniformly at random, and let $H$ be the subgraph of $H'$ induced by $\{x_1,\dots,x_n\}$. Then $H$ is a subgraph of $G$.
	Let $\N_L(H)$ denote the number of copies of $L$ in $H$.
	We have
	$$
	\mathbb{E}[e(H)] = \frac{e(H')}{b^k}
	$$
	and
	$$
	\mathbb{E}[\N_L(H)] = \frac{\N_L(H')}{b^{\ell}},
	$$
	as $\ell = |V(L)|$. By linearity of expectation, there is an outcome for $H$ satisfying
	\begin{equation}\label{eq:H}
	e(H) - \N_L(H) \geq \frac{e(H')}{b^k} - \frac{\N_L(H')}{b^{\ell}} \geq
	\ex_L(G) + \frac{N^{\ell-\varepsilon}}{b^{\ell}} - \frac{\N_L(H')}{b^{\ell}}.
	\end{equation}
	Note that since $H$ is a subgraph of $G$, we can obtain an $L$-free subgraph of $G$ by deleting one edge from every copy of $L$ in $H$. This deletes at most $\N_L(H)$ edges, and thus gives an $L$-free subgraph of $G$ with at least $e(H) - \N_L(H)$ edges. Thus, $\ex_L(G) \geq e(H) - \N_L(H)$. Combining this with \eqref{eq:H} and rearranging, we get
	$
	\N_L(H') \geq N^{\ell-\varepsilon},
	$
	as required.
	
	Having proved \eqref{eq:approx}, we can now complete the proof of the theorem. Suppose that there is an algorithm which approximates $\ex_{\mathcal{F}}(G')$ up to an additive error of $N^{k-\delta}$. Let $X$ be the number outputted by the algorithm, so
	$|\ex_{\mathcal{F}}(G') - X| \leq N^{k-\delta}$. Combined with \eqref{eq:approx}, we have
$$
\left| b^k\cdot \ex_L(G) - X \right| \leq N^{k-\delta} + \frac{N^{\ell-\varepsilon}}{b^{\ell-k}}\;.
$$
Dividing both sides by $b^k$, we see that
	$$
	\left| \ex_L(G) - \frac{X}{b^k} \right| \leq
	\frac{N^{k-\delta}}{b^k} + \frac{N^{\ell-\varepsilon}}{b^{\ell}} =
	\frac{n^{k-\delta}}{b^{\delta}} + \frac{n^{\ell-\varepsilon}}{b^{\varepsilon}} < \frac{1}{4} + \frac{1}{4} = \frac{1}{2},
	$$
	where the equality uses that $N = bn$, and the last inequality uses our choice of $b$. Thus, the above algorithm allows us to approximate $\ex_L(G)$ up to an additive error of less than $\frac{1}{2}$. As $\ex_L(G)$ is an integer, this allows us to compute $\ex_L(G)$ exactly. This completes the proof.
\end{proof}

\section{Proof of Lemma \ref{lem:triangle}}\label{sec:triangle reduction}
	We will use the two gadgets $\J_1$ and $\J_2$ depicted in Figure \ref{fig:gadgets}. Each of the gadgets has triangles $T_1,T_2,T_3$ and edges $e_1,e_2,e_3$, the latter called the {\em primary edges}. An edge is called {\em internal} if it does not belong to $T_1,T_2,T_3$.
	The gadgets have the following properties:
	\begin{figure}[h]
		\centering
		\begin{tikzpicture}[scale = 2.25]
		\coordinate (a1) at (0,0);
		\coordinate (a2) at (1,0);
		\coordinate (a3) at (0.5,0.866);
		\coordinate (a4) at (0.5,-0.5);
		\coordinate (a5) at ($0.5*(a2) + 0.5*(a3) + ({0.5*cos(30)},{0.5*sin(30)})$);
		\coordinate (a6) at ($0.5*(a1) + 0.5*(a3) + ({0.5*cos(150)},{0.5*sin(150)})$);
		\coordinate (a7) at ($0.5*(a3) + 0.5*(a5) + ({0.5*cos(75)},{0.5*sin(75)})$);
		\coordinate (a8) at ($0.5*(a3) + 0.5*(a6) + ({0.5*cos(105)},{0.5*sin(105)})$);
		\coordinate (a10) at ($0.5*(a5) + 0.5*(a7) + ({0.5*cos(15)},{0.5*sin(15)})$);
		\coordinate (a11) at ($0.5*(a6) + 0.5*(a8) + ({0.5*cos(165)},{0.5*sin(165)})$);
		
		\foreach \i in {1,2,3,4,5,6,7,8,10,11}
		{
			\draw (a\i) node[fill=black,circle,minimum size=2pt,inner sep=0pt] {};
		}
		
		\draw (a1) -- (a2) -- (a3) -- (a1);
		\draw (a2) -- (a5) -- (a3);
		\draw (a2) -- (a4) -- (a1);
		\draw (a1) -- (a6) -- (a3);
		\draw (a5) -- (a7) -- (a3);
		\draw (a3) -- (a8) -- (a6);
		\draw (a5) -- (a10) -- (a7);
		\draw (a6) -- (a11) -- (a8);	
		
		\filldraw[pattern=north east lines, pattern color=blue, draw opacity=0.4] (a1) -- (a2) -- (a3) -- cycle;
		\filldraw[pattern=north east lines, pattern color=blue, draw opacity=0.4] (a5) -- (a7) -- (a3) -- cycle;
		\filldraw[pattern=north east lines, pattern color=blue, draw opacity=0.4] (a3) -- (a8) -- (a6) -- cycle;
		\filldraw[pattern=north west lines, pattern color=red, draw opacity=0.4] (a2) -- (a5) -- (a3) -- cycle;
		\filldraw[pattern=north west lines, pattern color=red, draw opacity=0.4] (a1) -- (a6) -- (a3) -- cycle;
		
		\draw (a1) node[left] {\scriptsize $1$};
		\draw (a2) node[right] {\scriptsize $2$};
		\draw (a3) node[above] {\scriptsize $3$};
		\draw (a4) node[left] {\scriptsize $3$};
		\draw (a5) node[right] {\scriptsize $1$};
		\draw (a6) node[left] {\scriptsize $2$};
		\draw (a7) node[left] {\scriptsize $2$};
		\draw (a8) node[right] {\scriptsize $1$};
		\draw (a10) node[above] {\scriptsize $3$};
		\draw (a11) node[above] {\scriptsize $3$};
		
		\draw ($0.5*(a5) + 0.5*(a7) + ({0.08*cos(15+180)},{0.08*sin(15+180)})$) node {\scriptsize $e_1$};
		\draw ($0.5*(a6) + 0.5*(a8) + ({0.08*cos(165+180)},{0.08*sin(165+180)})$) node {\scriptsize $e_2$};
		\draw ($0.5*(a1) + 0.5*(a2) + ({0.05*cos(90)},{0.05*sin(90)})$) node {\scriptsize $e_3$};
		\draw ($0.33*(a5) + 0.33*(a7) + 0.33*(a10)$) node {\footnotesize $T_1$};
		\draw ($0.33*(a6) + 0.33*(a8) + 0.33*(a11)$) node {\footnotesize $T_2$};
		\draw ($0.3*(a1) + 0.3*(a2) + 0.4*(a4)$) node {\footnotesize $T_3$};
		\end{tikzpicture}
		\hspace{2.5cm}
		\begin{tikzpicture}[scale = 2.25]
		\coordinate (a1) at (0,0);
		\coordinate (a2) at (1,0);
		\coordinate (a3) at (0.5,0.866);
		\coordinate (a4) at (0.5,-0.5);
		\coordinate (a5) at ($0.5*(a2) + 0.5*(a3) + ({0.5*cos(30)},{0.5*sin(30)})$);
		\coordinate (a6) at ($0.5*(a1) + 0.5*(a3) + ({0.5*cos(150)},{0.5*sin(150)})$);
		\coordinate (a7) at ($0.5*(a3) + 0.5*(a5) + ({0.5*cos(75)},{0.5*sin(75)})$);
		\coordinate (a8) at ($0.5*(a3) + 0.5*(a6) + ({0.5*cos(105)},{0.5*sin(105)})$);
		\coordinate (a9) at ($0.5*(a2) + 0.5*(a4) + ({0.5*cos(-45)},{0.5*sin(-45)})$);
		\coordinate (a10) at ($0.5*(a5) + 0.5*(a7) + ({0.5*cos(15)},{0.5*sin(15)})$);
		\coordinate (a11) at ($0.5*(a6) + 0.5*(a8) + ({0.5*cos(165)},{0.5*sin(165)})$);
		\coordinate (a12) at ($0.5*(a4) + 0.5*(a9) + ({0.5*cos(-105)},{0.5*sin(-105)})$);
		
		\foreach \i in {1,2,3,4,5,6,7,8,9,10,11,12}
		{
			\draw (a\i) node[fill=black,circle,minimum size=2pt,inner sep=0pt] {};
		}
		
		\draw (a1) -- (a2) -- (a3) -- (a1);
		\draw (a2) -- (a5) -- (a3);
		\draw (a2) -- (a4) -- (a1);
		\draw (a1) -- (a6) -- (a3);
		\draw (a5) -- (a7) -- (a3);
		\draw (a3) -- (a8) -- (a6);
		\draw (a2) -- (a9) -- (a4);
		\draw (a5) -- (a10) -- (a7);
		\draw (a6) -- (a11) -- (a8);
		\draw (a4) -- (a12) -- (a9);
		
		\filldraw[pattern=north east lines, pattern color=blue, draw opacity=0.4] (a1) -- (a2) -- (a3) -- cycle;
		\filldraw[pattern=north east lines, pattern color=blue, draw opacity=0.4] (a5) -- (a7) -- (a3) -- cycle;
		\filldraw[pattern=north east lines, pattern color=blue, draw opacity=0.4] (a3) -- (a8) -- (a6) -- cycle;
		\filldraw[pattern=north east lines, pattern color=blue, draw opacity=0.4] (a2) -- (a9) -- (a4) -- cycle;
		
		\filldraw[pattern=north west lines, pattern color=red, draw opacity=0.4] (a2) -- (a5) -- (a3) -- cycle;
		\filldraw[pattern=north west lines, pattern color=red, draw opacity=0.4] (a2) -- (a4) -- (a1) -- cycle;
		\filldraw[pattern=north west lines, pattern color=red, draw opacity=0.4] (a1) -- (a6) -- (a3) -- cycle;
		
		\draw (a1) node[left] {\scriptsize $1$};
		\draw (a2) node[right] {\scriptsize $2$};
		\draw (a3) node[above] {\scriptsize $3$};
		\draw (a4) node[left] {\scriptsize $3$};
		\draw (a5) node[right] {\scriptsize $1$};
		\draw (a6) node[left] {\scriptsize $2$};
		\draw (a7) node[left] {\scriptsize $2$};
		\draw (a8) node[right] {\scriptsize $1$};
		\draw (a9) node[right] {\scriptsize $1$};
		\draw (a10) node[above] {\scriptsize $3$};
		\draw (a11) node[above] {\scriptsize $3$};
		\draw (a12) node[left] {\scriptsize $2$};
		
		\draw ($0.5*(a5) + 0.5*(a7) + ({0.08*cos(15+180)},{0.08*sin(15+180)})$) node {\scriptsize $e_1$};
		\draw ($0.5*(a6) + 0.5*(a8) + ({0.08*cos(165+180)},{0.08*sin(165+180)})$) node {\scriptsize $e_2$};
		\draw ($0.5*(a4) + 0.5*(a9) + ({0.06*cos(-105+180)},{0.06*sin(-105+180)})$) node {\scriptsize $e_3$};
		\draw ($0.33*(a5) + 0.33*(a7) + 0.33*(a10)$) node {\footnotesize $T_1$};
		\draw ($0.33*(a6) + 0.33*(a8) + 0.33*(a11)$) node {\footnotesize $T_2$};
		\draw ($0.33*(a4) + 0.33*(a9) + 0.33*(a12)$) node {\footnotesize $T_3$};
		\end{tikzpicture}
		
		\caption{The two gadgets $\J_1$ (left) and $\J_2$ (right) used in the proof of Lemma \ref{lem:triangle}. The gadget $\J_1$ corresponds to the clause $(x_1 \vee x_2 \vee x_3)$ and the gadget $\J_2$ corresponds to the clause $(x_1 \vee x_2 \vee \overline{x_3})$.}\label{fig:gadgets}
	\end{figure}
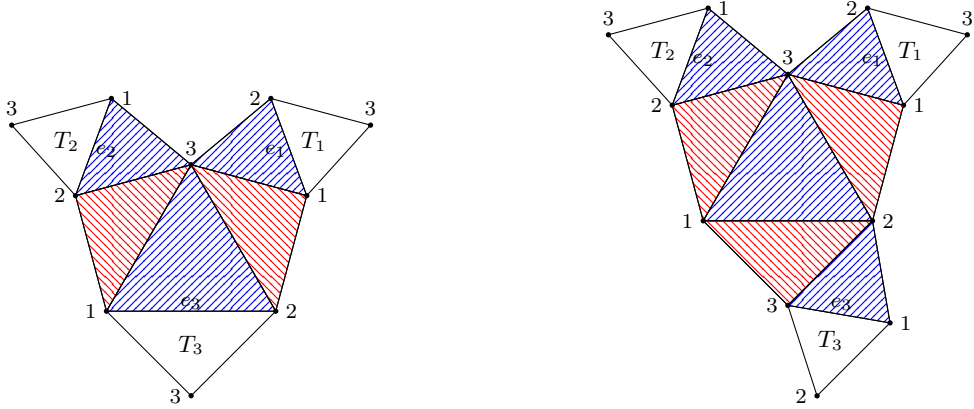

	\begin{claim}\label{claim:gadget correctness}
		Set $b_1 := 2$ and $b_2 := 3$. The following holds for $i=1,2$.
		\begin{enumerate}
			\item Let $E$ be a set of edges of $\J_i$ such that $\J_i - E$ is triangle-free. Then $E$ contains at least
			$b_i$ internal edges. Moreover, if $e_1,e_2,e_3 \notin E$, then $E$ contains at least $b_i+1$ internal edges.
			\item On the other hand, for every choice of edges $\hat{e_j} \in E(T_j)$ for $j=1,2,3$, such that $\hat{e_j} = e_j$ for at least one $j$,
			there exists a set $E \subseteq E(\J_i)$ such that $\J_i - E$ is triangle-free, and $E$ consists of $\hat{e_1},\hat{e_2},\hat{e_3}$ and exactly $b_i$ internal edges of $\J_i$.
		\end{enumerate}
	\end{claim}
	\begin{proof}
		Let $i = 1,2$. Figure \ref{fig:gadgets} shows that $\J_i$ has $b_i$ edge-disjoint triangles (colored red) consisting only of internal edges. Thus, one has to delete at least $b_i$ edges to destroy all triangles in $\J_i$. Also, Figure \ref{fig:gadgets} shows that $\J_i$ has $b_i+1$ edge-disjoint triangles (colored blue) which consist only of internal edges and $e_1,e_2,e_3$. Thus, if we do not delete $e_1,e_2,e_3$, then we must delete at least $b_i+1$ internal edges to destroy all triangles in $\J_i$. This proves Item 1. Item 2 can be verified by hand; namely, it can be verified that for every given $1 \leq j \leq 3$, there is a choice of one edge from each red triangle such that the chosen $b_i$ edges together with $e_j$ cover all triangles which are not $T_1,T_2,T_3$. This implies Item 2, because the triangles $T_1,T_2,T_3$ are destroyed by deleting the edges $\hat{e_1},\hat{e_2},\hat{e_3}$.
	\end{proof}
	
	\begin{claim}\label{claim:gadget 3-coloring}
		Let $a_1,a_2,a_3 \in \binom{[3]}{2}$.
		\begin{enumerate}
			\item If $a_1=a_2=a_3$, then there is a proper 3-coloring $f : V(\J_1) \rightarrow [3]$ such that\footnote{Here we use the notation $f(e) := \{f(x) : x \in e\}$.} $f(e_i) = a_i$ for $i=1,2,3$.
			\item If $a_1=a_2$ and $a_3 \neq a_1,a_2$, then there is a proper $3$-coloring $f : V(\J_2) \rightarrow [3]$ such that $f(e_i) = a_i$ for $i=1,2,3$.
		\end{enumerate}
	\end{claim}
	\begin{proof}
		Such 3-colorings are depicted in Figure \ref{fig:gadgets}. For Item 1, Figure \ref{fig:gadgets} shows a 3-coloring $f$ of $\J_1$ with $f(e_i) = \{1,2\}$ for $i=1,2,3$. By changing the names of the colors, we can get a coloring with $f(e_i) = a$ ($i=1,2,3$) for any $a \in \binom{[3]}{2}$. Similarly, for Item 2, Figure \ref{fig:gadgets} shows a 3-coloring $f$ of $\J_2$ with $f(e_1) = f(e_2) = \{1,2\}$ and $f(e_3) = \{1,3\}$. Changing the names of the colors proves Item 2.
	\end{proof}

We now use the above two claims to prove Lemma \ref{lem:triangle}. We reduce the problem 3-CNF SAT (which is well-known to be NP-hard) to the problem of computing $\rem_2(G,\P,K_3)$. Let $\Phi$ be a 3-CNF formula with variables $x_1,\dots,x_n$ and clauses $C_1,\dots,C_m$. So each clause $C_{\ell}$ is of the form $v_i \vee v_j \vee v_k$ for some $1 \leq i < j < k \leq n$, where $v_i \in \{x_i,\overline{x_i}\}$ and similarly for $v_j,v_k$.
	We construct a 3-partite graph $G$ as follows. First, take $n$ vertex-disjoint triangles $T_1,\dots,T_n$ and color their vertices properly with $1,2,3$. Triangle $T_i$ corresponds to variable $x_i$. Later on, if we shall choose to delete the edge of $T_i$ colored $\{1,2\}$ (when turning $G$ into a triangle-free graph), then this will correspond to assigning $x_i$ to be true, and if we shall choose to delete the edge colored $\{1,3\}$, then this will correspond to assigning $x_i$ to be false.
	
	Next, let $1 \leq \ell \leq m$, and let $x_i,x_j,x_k$ be the variables appearing in $C_{\ell}$.
	We say that $C_{\ell}$ is of type 1 if the three variables $x_i,x_j,x_k$ all appear without negation or all appear negated. Else, namely, if two of the variables $x_i,x_j,x_k$ have the same sign and the third one has the opposite sign, then we say that $C_{\ell}$ is of type 2.
	Set $a_i := \{1,2\}$ if $x_i$ appears in $C_{\ell}$ (without negation), and set $a_i := \{1,3\}$ if $\overline{x_i}$ appears in $C_{\ell}$. Define $a_j,a_k \in \{\{1,2\},\{1,3\}\}$ analogously (with respect to $x_j,x_k$, respectively).
If $C_{\ell}$ is of type 1 (so that $a_i = a_j = a_k$), then add a copy $J_{\ell}$ of $\J_1$ in which $T_i,T_j,T_k$ play the roles of $T_1,T_2,T_3$, respectively; the primary edges are the edges in $T_i,T_j,T_k$ of colors $a_i,a_j,a_k$, respectively; and all other vertices are new.
For example, the left part of Figure \ref{fig:gadgets} depicts the copy $J_{\ell}$ of $\J_1$ when $C_{\ell}$ is the clause $(x_1 \vee x_2 \vee x_3)$.
Similarly, suppose that $C_{\ell}$ is of type 2 and assume, without loss of generality, that $a_i = a_j$ and $a_k \neq a_i,a_j$.
Add a copy $J_{\ell}$ of $\J_2$ in which $T_i,T_j,T_k$ play the roles of $T_1,T_2,T_3$, respectively; the primary edges are the edges in $T_i,T_j,T_k$ of colors $a_i,a_j,a_k$, respectively; and all other vertices are new.
For example, the right part of Figure \ref{fig:gadgets} depicts the copy $\J_2$ when $C_{\ell}$ is the clause $(x_1 \vee x_2 \vee \overline{x_3})$.
By Claim \ref{claim:gadget 3-coloring}, in either of the cases (namely, independently of the type of $C_{\ell}$), we can extend the $3$-coloring of $T_i,T_j,T_k$ into a proper $3$-coloring of $J_{\ell}$.
	
	The resulting graph is $G$. By the above, $G$ is $3$-colorable, and a $3$-coloring of $G$ can be explicitly specified. Let $m_0$ be the number of clauses $C_{\ell}$ ($1 \leq \ell \leq m$) of type 1; so $m-m_0$ clauses are of type 2.
	If $C_{\ell}$ is of type $1$ then set $b_{\ell} := 2$, and if $C_{\ell}$ is of type 2 then set $b_{\ell} := 3$ (this is in accordance with Claim \ref{claim:gadget correctness}). Note that $\sum_{\ell=1}^m b_{\ell} = n + 3m - m_0 = 2m_0 + 3(m-m_0) = 3m - m_0$.
	
	To complete the proof, we claim that
	$\rem_2(G,\P,K_3) = n + 3m - m_0$ if and only if $\Phi$ is satisfiable. First, we claim that $\rem_2(G,\P,K_3) \geq n + 3m - m_0$ (regardless of whether or not $\Phi$ is satisfiable).
	Indeed, in order to make $G$ triangle-free, we have to delete at least one edge from $T_i$ for every $i=1,\dots,n$, and we also have to delete at least $b_{\ell}$ internal edges of $J_{\ell}$ for every $1 \leq \ell \leq m$, by Item 1 of Claim \ref{claim:gadget correctness}. Thus, we must delete at least $n + \sum_{\ell=1}^m b_{\ell} = n + 3m - m_0$ edges, as claimed.
	Now, suppose that $\rem_2(G,\P,K_3) = n + 3m - m_0$, and fix a minimal-size set $E \subseteq E(G)$ such that $G-E$ is triangle-free. As $|E| = n + 3m - m_0$, the above lower bound on $\rem_2(G,\P,K_3)$ is tight, meaning that $E$ must contain exactly one edge from $T_i$ for every $1 \leq i \leq n$ and exactly $b_{\ell}$ internal edges from $J_{\ell}$ for every $1 \leq \ell \leq m$.
	Also, we may assume that for every $1 \leq i \leq n$, $E$ does not contain the $\{2,3\}$-colored edge of $T_i$. Indeed, otherwise, we can replace this edge with the $\{1,2\}$- or $\{1,3\}$-colored edge of this $T_i$, and the resulting $E$ will still intersect every triangle of $G$, because the $\{2,3\}$-colored edge of $T_i$ participates in only one triangle in $G$, namely $T_i$.
	Now, set variable $x_i$ to be true if the $\{1,2\}$-colored edge of $T_i$ belongs to $E$, and set $x_i$ to be false if the $\{1,3\}$-colored edge of $T_i$ belongs to $E$.
	Fix any $1 \leq \ell \leq m$ and suppose that $C_{\ell}$ has variables $x_i,x_j,x_k$. By the ``moreover"-part of Item 1 in Claim \ref{claim:gadget correctness}, the set $E$ must contain one of the primary edges of $J_{\ell}$. Without loss of generality, suppose that $E$ contains the primary edge belonging to $T_i$. By the definition of $J_{\ell}$, this edge is the $\{1,2\}$-edge of $T_i$ if $x_i$ appears in $C_{\ell}$ and the $\{1,3\}$-edge of $T_i$ if $\overline{x_i}$ appears in $C_{\ell}$. The fact that this edge belongs to $E$, and the way we assigned a truth-value to $x_i$, imply that $C_{\ell}$ is satisfied.
	
	In the other direction, suppose that $\Phi$ is satisfiable, and fix a satisfying assignment to $x_1,\dots,x_n$. We define an edge-set $E \subseteq E(G)$ as follows. For each $i = 1,\dots,n$, place the $\{1,2\}$-edge (resp. $\{1,3\}$-edge) of $T_i$ into $E$ if $x_i$ is true (resp. false); denote this chosen edge by $\hat{e_i}$.
	Now fix $1 \leq \ell \leq m$ and suppose $C_{\ell}$ has variables $x_i,x_j,x_k$. Without loss of generality, suppose that the variable $x_i$ satisfies $C_{\ell}$.
	Thus, if $x_i$ appears in $C_{\ell}$, then $\hat{e_i}$ is the $\{1,2\}$-edge of $T_i$, and if $\overline{x_i}$ appears in $C_{\ell}$, then $\hat{e_i}$ is the $\{1,3\}$-edge of $T_i$.
	By the definition of $J_{\ell}$, this means that $\hat{e_i}$ is the primary edge belonging to $T_i$ in the gadget $J_{\ell}$.
	Now, by Item 2 of Claim \ref{claim:gadget correctness}, there exists a set $E_{\ell} \subseteq E(J_{\ell})$ such that $J_{\ell} - E_{\ell}$ is triangle-free, and $E_{\ell}$ consists of the edges $\hat{e_i},\hat{e_j},\hat{e_k}$ and of exactly $b_{\ell}$ internal edges of $J_{\ell}$. Add $E_{\ell}$ to $E$. Doing this for every $\ell = 1,\dots,m$ gives a set $E \subseteq E(G)$ such that $G - E$ is triangle-free and $|E| = n + \sum_{\ell=1}^{m}b_{\ell} = n+3m-m_0$, as required.
	This completes the proof of the lemma.

\section{Proof of Theorem \ref{thm:matching}}\label{sec:matching}

Let $M_r$ denote the $k$-uniform matching of size $r$ (throughout this section, we consider $k$-uniform hypergraphs, so we omit $k$ from the notation), and let $G$ be a $k$-uniform hypergraph.
As we mentioned in the introduction, our goal is to devise an inductive process for generating
the set $\mathcal{M}_r(G)$, which is the set of maximal (with respect to inclusion) subgraphs of $G$ without a matching of size $r$. This process is described in Lemma \ref{lem:main1} below.
In what follows, we use $\nu(G)$ to denote the matching number of $G$, namely the size of a largest matching in $G$.
For a subset $S \subseteq V(G)$, let $L_G(S)$ denote the {\em link} of $S$, i.e., $L_G(S)$ is the $(k-|S|)$-graph on $V(G)$ with edge-set $\{e \setminus S : e \in E(G), S \subseteq e\}$. For an integer $t$, a set $S$ is called {\em $t$-heavy} in $G$ if $\nu(L_G(S)) > t$, and otherwise $S$ is called {\em $t$-light}. We start with the following simple yet useful observation.

\begin{lemma}\label{lem:heavy set}
Let $G$ be a $k$-graph and let $F$ be a subgraph of $G$ with $\nu(F) < r$. Let $S \subseteq V(G)$ such that $S$ is $t$-heavy in $F$, where $t = (r-1)k$. Let $F'$ be the $k$-graph obtained from $F$ by adding all edges of $G$ containing $S$. Then $\nu(F') < r$.
\end{lemma}
\begin{proof}

Suppose by contradiction that there is a matching $e_1, \dots, e_r$ in $F'$. Since $\nu(F) < r$, one of the edges $e_i$ must be in $E(F') \setminus E(F)$, meaning that it contains $S$.
Without loss of generality, assume that $S \subseteq e_1$. 
As $e_1,\dots,e_r$ are pairwise disjoint, we get that $S$ is disjoint from $e_2,\dots,e_r$, and hence $e_2,\dots,e_r \in E(F)$. 
Next, as $S$ is $t$-heavy in $F$, there exists a matching $f_1, \dots, f_{t+1}$ in $L_F(S)$. Note that $|e_2 \cup \dots \cup e_r| = (r-1)k$, and so, the set $e_2 \cup \dots \cup e_r$ intersects at most $(r-1)k$ of the edges $f_1, \dots, f_{t+1}$. As $t+1 > (r-1)k$, there is $1 \leq i \leq t+1$ such that $f_i$ is disjoint from
$e_2, \dots, e_r$. We now get that $f_i \cup S, e_2, \dots, e_r$ is a matching of size $r$ in $F$ -- a contradiction.
\end{proof}
		

We are now ready to describe the inductive statement of our process. Recall that $\mathcal{M}_r(G)$ denotes the set of inclusion-wise maximal subgraphs of $G$ without a matching of size $r$

\begin{lemma}\label{lem:main1}
Given a $k$-graph $G$, one can generate in time $\poly(n)$, for every $0 \leq i \leq k$,
a family\footnote{In particular, each family $\mathcal{H}_i$ has polynomial size.} $\mathcal{H}_i$ of pairs $(H,A_H)$ so that $H$ is a spanning subgraph of $G$, and $A_H \subseteq V(G) = V(H)$ is of size $|A_H| = O(1)$, such that the following holds. For every $F \in \mathcal{M}_r(G)$, there is $(H,A_H) \in \mathcal{H}_i$ such that for every $e \in E(G)$ with $|e \cap A_H| \leq i$, it holds that $e \in E(F)$ if and only if $e \in E(H)$.
\end{lemma}
Note (crucially) that if we take
$i=k$ in the above lemma, then we are guaranteed that for every $F \in \mathcal{M}_r(G)$, there is a pair $(H,A_H) \in \mathcal{H}_k$ such that for every $e \in E(G)$ with $|e \cap A_H| \leq k$ (namely, for every $e \in E(G)$), we have $e \in E(F)$ if and only if $e \in E(H)$. Thus $H = F$, meaning that $\mathcal{H}_k$ contains a pair of the form $(F,A_F)$. So indeed $\mathcal{M}_r(G) \subseteq \mathcal{H}_k$.

\begin{proof}[Proof of Lemma \ref{lem:main1}]
The proof is by induction on $i$. For the base case $i=0$, take $\mathcal{H}_0$ to consist of all pairs $(M,A_M)$
where $M$ is a matchings in $G$ of size at most $r-1$ and $A_M$ is the vertex set of $M$. Clearly, all such pairs can be enumerated in time $O(n^{(r-1)k})$,
and all sets $A_M$ are of size at most $(r-1)k=O(1)$. Now, fix any $F \in \mathcal{M}_r(G)$, and let $M$ be a maximal matching in $F$ (so $|M| \leq r-1$).
By the definition of $\mathcal{H}_0$, we have $(M,A_M) \in \mathcal{H}_0$ (where $A_M$ is the set of vertices covered by $M$).
Now it is clear that for every $e \in E(G)$ satisfying $|e \cap A_M| = 0$, we have $e \notin M$ and $e \notin E(F)$, since $M$ is a maximal matching in $F$.

		
Now let $1 \leq i \leq k$. By the induction hypothesis, one can generate in time $\poly(n)$ a family $\mathcal{H}_{i-1}$ with the properties stated in the lemma. To define $\mathcal{H}_i$, we proceed as follows. Set $t := (r-1)k$.
Go over all $(H,A_H) \in \mathcal{H}_{i-1}$, and for each $H$, go over all families of sets $\mathcal{F} \subseteq \binom{A_H}{i}$.
For each such family of sets ${\cal F}$, let $H_{\cal F}$ be the $k$-graph consisting of the edges of $H$, together with all edges of $G$ containing one of the sets in $\binom{A_H}{i} \setminus \mathcal{F}$.
Note that each $H \in \mathcal{H}_{i-1}$ generates at most $2^{\binom{|A_H|}{i}}=O(1)$ graphs $H_{\cal F}$.
Now, for each ${\cal F}$ as above, go over all functions $f : \mathcal{F} \rightarrow 2^{E(G)}$ with the property that for every $S \in \mathcal{F}$ it holds that $|f(S)| \leq t$; $S \subseteq e$ for every $e \in f(S)$; and $\{e \setminus S : e \in f(S)\}$ is a matching.
Namely, the function $f$ chooses, for each $S \in \mathcal{F}$, a matching of size at most $t$ in $L_G(S)$.
Finally, let $A_{\mathcal{F},f}$ be the union of $A_H$ and all edges $e$ for $e \in f(S), S \in \mathcal{F}$.
Then $|A_{\mathcal{F},f}| \leq |A_H| + \binom{|A_H|}{i} \cdot t(k-i) = O(1)$.
The family $\mathcal{H}_i$ then contains all pairs $(H_{\cal F},A_{\mathcal{F},f})$.
Note that for every graph $H_{\cal F}$ as defined above, we put in $\mathcal{H}_{i}$ at most $O\left( n^{\binom{|A_H|}{i} \cdot t(k-i)} \right)$ pairs\footnote{We stress that each ${\cal F}$ defines a single graph $H_{\cal F}$ but several sets $A_{\mathcal{F},f}$ (one for each $f$). Hence, $\mathcal{H}_i$ contains several pairs $(H,A_H)$ with the same $H$.} $(H_{\cal F},A_{\mathcal{F},f})$, one for each function $f$ as above.
It is thus clear that if $\mathcal{H}_{i-1}$ can be generated in time $\poly(n)$ then so can $\mathcal{H}_{i}$.
		

It remains to show that $\mathcal{H}_i$ has the desired property. So fix any
$F \in \mathcal{M}_r(G)$. By the induction hypothesis, there is $(H,A_H) \in \mathcal{H}_{i-1}$ such that for every $e \in E(G)$ with $|e \cap A_H| \leq i-1$, it holds that $e \in E(F)$ if and only if $e \in E(H)$. Let $\mathcal{F}$ be the set of all $i$-tuples $S \in \binom{A_H}{i}$ such that $S$ is $t$-light in $F$. Let $H_{\cal F}$ be as defined in the previous paragraph, i.e., $H_{\cal F}$ consists of the edges of $H$ and all edges of $G$ which contain a set in $\binom{A_H}{i} \setminus \mathcal{F}$.
For each $S \in \mathcal{F}$, let $M(S)$ be a maximal matching in $L_{F}(S)$; so $|M(S)| \leq t$ because $S$ is $t$-light. Let $f(S) := \{S \cup e : e \in M(S)\}$.
Let $A_{\mathcal{F},f}$ be as defined in the previous paragraph, recalling that $A_{\mathcal{F},f}$ is the union of $A_H$ and all edges $e$ with $e \in f(S), S \in \mathcal{F}$.
We claim that the pair $(H_{\cal F},A_{\mathcal{F},f})$ satisfies the desired property of the lemma, namely, that for every edge $e \in E(G)$, if $|e \cap A_{\mathcal{F},f}| \leq i$, then $e \in E(F)$ if and only if $e \in E(H_{\cal F})$. Suppose first that $|e \cap A_H| \leq i-1$. Then $e \in E(F)$ if and only if $e \in E(H)$, by the choice of $H$. Also, observe that $e \in E(H)$ if and only if $e \in E(H_{\cal F})$, because every edge in $E(H_{\cal F}) \setminus E(H)$ intersects $A_H$ in at least $i$ vertices. Thus, $e \in E(F)$ if and only if $e \in E(H_{\cal F})$.
		
Now suppose that $|e \cap A_H| \geq i$. Then, as $|e \cap A_{\mathcal{F},f}| \leq i$ and $A_H \subseteq A_{\mathcal{F},f}$, we must have $|e \cap A_H| = i$ and $S := e \cap A_H = e \cap A_{\mathcal{F},f}$. There are now two cases: Suppose first that $S \in \mathcal{F}$, i.e., $S$ is $t$-light in $F$. Let $M(S)$ be the maximal matching in $L_{F}(S)$ we have chosen above. Since $S \subseteq e$, we have that $e \setminus S$ is an edge in $L_{F}(S)$, and as $M(S)$ is a maximal matching, we have that $(e \setminus S) \cap V(M(S)) \neq \emptyset$. On the other hand, we have $V(M(S)) \subseteq A_{\mathcal{F},f}$ by the definition of $A_{\mathcal{F},f}$. It follows that $|e \cap A_{\mathcal{F},f}| > |S| = i$, a contradiction.
So it remains to consider the case where $S$ is $t$-heavy in $F$. By Lemma \ref{lem:heavy set}, adding to $F$ all edges of $G$ containing $S$ does not create a matching of size $r$. Recall that $F \in \mathcal{M}_r(G)$, i.e., $F$ is an inclusion-wise maximal subgraph of $G$ with no matching of size $r$. Thus, $F$ must contain all edges of $G$ containing $S$. In particular, $e \in E(F)$. Also, $e \in E(H_{\cal F})$ since $S \in \binom{A_H}{i} \setminus \mathcal{F}$. Thus, $e \in E(F)$ if and only if $e \in E(H_{\cal F})$, as required. This completes the proof of the lemma.
\end{proof}
\noindent



\begin{proof}[Proof of Theorem \ref{thm:matching}]
Let $\mathcal{H}_k$ be the family of pairs $(H,A_H)$ generated by Lemma \ref{lem:main1} (in time $\poly(n)$). To compute $\ex_{M_r}(G)$, we compute the maximum value of $e(H)$ over all graphs $H$ belonging to the pairs in $\mathcal{H}_k$ which are $M_r$-free. Denote this maximum by $m$. Note that $m$ can be computed in polynomial time because $\mathcal{H}_k$ is of polynomial size, and checking whether a $k$-graph $H$ is $M_r$-free can clearly be done in time $O(n^{kr})$. We now claim that $m = \ex_{M_r}(G)$. Clearly, $m \leq \ex_{M_r}(G)$, because the maximum in the definition of $m$ is taken over $M_r$-free subgraphs of $G$. On the other hand, $\ex_{M_r}(G) = \max_{F \in \mathcal{M}_r(G)}e(F)$ holds by the definition of $\mathcal{M}_r(G)$. As we noted immediately after the statement of Lemma \ref{lem:main1}, for every $F \in \mathcal{M}_r(G)$
we have a pair of the form $(F,A_F)$ in $\mathcal{H}_k$. Thus, $m \geq \ex_{M_r}(G)$. This completes the proof of the theorem.
\end{proof}

\appendix

\section{Proof of Proposition \ref{prop1}}

For integers $k \geq 3, \, k > t \geq 1$, let $E_t^{(k)}$ denote a pair of $k$-edges that intersect in exactly $t$ vertices.

\begin{theorem}\label{thm:two-intersecting-edges-hardness}
	For any integers $k \geq 3, \, k > t \geq 1$, it is NP-hard to compute $\ex(G,E_t^{(k)})$ for an input $k$-uniform hypergraph $G$ on $n$ vertices.
\end{theorem}

\noindent We begin by proving the following two simple propositions.
\begin{proposition}\label{prop:ex(G,k,t)<=ex(G',k+1,t)}
	For any integers $k \geq 3, \, k > t \geq 1$, computing $\ex(G',E_t^{(k+1)})$ is at least as hard as
	computing $\ex(G,E_t^{(k)})$ for $n$-vertex input hypergraphs.
\end{proposition}
\begin{proof}
	Let $G$ be a $k$-uniform hypergraph. We reduce the problem of computing $\ex(G,E_t^{(k)})$ to the problem of computing $\ex(G',E_t^{(k+1)})$, where $G'$ is the following $(k+1)$-uniform hypergraph: For every edge $e \in E(G)$, add to $e$ a new vertex $v_e$. The resulting graph is $G'$, which is clearly $(k+1)$-uniform.
	We now show that $\ex(G,E_t^{(k)}) = \ex(G',E_t^{(k+1)})$.
	In one direction, let $F$ be an $E_t^{(k)}$-free subgraph of $G$ with $|E(F)| = \ex(G,E_t^{(k)})$.
	Let $F'$ be the subgraph of $G'$ whose edge-set is $\{e \cup \{v_e\} : e \in E(F)\}$.
	It is straightforward to see that $F'$ is $E_t^{(k+1)}$-free, and so $\ex(G',E_t^{(k+1)}) \geq |E(F')| = |E(F)| = \ex(G,E_t^{(k)})$.
	
	In the other direction, let $F'$ be an $E_t^{(k+1)}$-free subgraph of $G'$ with $|E(F')| = \ex(G',E_t^{(k+1)})$.
	Let $F$ be the subgraph of $G$ whose edge-set is $\{e \setminus \{v_e\} : e \in E(F')\}$.
	It is straightforward to see that $F$ is $E_t^{(k)}$-free, and so $\ex(G,E_t^{(k)}) \geq |E(F)| = |E(F')| = \ex(G',E_t^{(k+1)})$.		
\end{proof}

\begin{proposition}\label{prop:ex(G,k,t)<=ex(G',k+1,t+1)}
	For any integers $k \geq 3, \, k > t \geq 1$, computing $\ex(G',E_{t+1}^{(k+1)})$ is at least as hard as
	computing $\ex(G,E_t^{(k)})$ for $n$-vertex input hypergraphs.
\end{proposition}
\begin{proof}
	Let $G$ be a $k$-uniform hypergraph. We reduce the problem of computing $\ex(G,E_t^{(k)})$ to the problem of computing
	$\ex(G',E_{t+1}^{(k+1)})$, where $G'$ is the following $(k+1)$-uniform hypergraph: Add a new vertex $v$ to $V(G)$, and then add $v$ to every edge $e \in E(G)$. The resulting graph is $G'$, which is clearly $(k+1)$-uniform.
	We now show that $\ex(G,E_t^{(k)}) = \ex(G',E_{t+1}^{(k+1)})$.
	In one direction, let $F$ be an $E_t^{(k)}$-free subgraph of $G$ with $|E(F)| = \ex(G,E_t^{(k)})$.
	Let $F'$ be the subgraph of $G'$ whose edge-set is $\{e \cup \{v\} : e \in E(F)\}$.
	It is straightforward to see that $F'$ is $E_{t+1}^{(k+1)}$-free, and so $\ex(G',E_{t+1}^{(k+1)}) \geq |E(F')| = |E(F)| = \ex(G,E_t^{(k)})$.
	
	In the other direction, let $F'$ be an $E_{t+1}^{(k+1)}$-free subgraph of $G'$ with $|E(F')| = \ex(G',E_{t+1}^{(k+1)})$.
	Let $F$ be the subgraph of $G$ whose edge-set is $\{e \setminus \{v\} : e \in E(F')\}$.
	It is straightforward to see that $F$ is $E_t^{(k)}$-free, and so $\ex(G,E_t^{(k)}) \geq |E(F)| = |E(F')| = \ex(G',E_{t+1}^{(k+1)})$.
\end{proof}

\noindent In the following lemma we prove Theorem \ref{thm:two-intersecting-edges-hardness} for the basic case of $k=3, \, t=2$. We then combine this lemma with Propositions \ref{prop:ex(G,k,t)<=ex(G',k+1,t)} and \ref{prop:ex(G,k,t)<=ex(G',k+1,t+1)} in order to prove the theorem in its full generality.	
\begin{lemma}\label{lem:(3,2)-hardness}
	It is NP-hard to compute $\ex(G,E_2^{(3)})$ for an input $3$-uniform hypergraph $G$ on $n$ vertices.
\end{lemma}

\noindent In our proof we use a reduction from a variation of the MAX-2-SAT problem. We first recall the definition of the original problem.
\begin{definition}\label{def:max-2-sat}
	A {\em 2-CNF formula} is a formula that consists of a conjunction of clauses, where each clause is a disjunction of $2$ literals\footnote{Clauses of the form $(x \lor x) \equiv (x)$ are not allowed, and we may assume that we do not have clauses of the form $(x \lor \overline{x})$.} (a variable or its negation).
	The {\em MAX-2-SAT problem} is to determine, given a 2-CNF formula $\varphi$, the maximum number of clauses that can be satisfied in $\varphi$ by a truth assignment to its variables.
\end{definition}

\noindent It is well-known that the MAX-2-SAT problem is NP-hard (see \cite{GJ}).
We consider the following variation of the MAX-2-SAT problem, which we call {\em 3-OCC-MAX-2-SAT}.
\begin{definition}\label{def:3-occ-max-2-sat}
	We call a 2-CNF formula in which every variable occurs at most $3$ times a {\em 3-OCC-2-SAT formula}.
	The {\em 3-OCC-MAX-2-SAT problem} is to determine, given a 3-OCC-2-SAT formula $\varphi$, the maximum number of clauses that can be satisfied in $\varphi$ by a truth assignment to its variables.
\end{definition}
\noindent It was shown by Berman and Karpinski \cite{Ber-Karp} that the 3-OCC-MAX-2-SAT problem is also NP-hard\footnote{In fact, Berman and Karpinski showed that for every $\varepsilon > 0$, it is NP-hard to approximate 3-OCC-MAX-2-SAT within a factor of $2012/2011 - \varepsilon$ (see Section 7 in \cite{Ber-Karp}).}.
\begin{theorem}[Berman and Karpinski \cite{Ber-Karp}]\label{thm:3-OCC-MAX-2-SAT-NP-hard}
	The 3-OCC-MAX-2-SAT problem is NP-hard.
\end{theorem}

\noindent
We are now ready to prove Lemma \ref{lem:(3,2)-hardness}.

\begin{proof}[Proof of Lemma \ref{lem:(3,2)-hardness}]
	We show that computing $\ex(G,E_2^{(3)})$ for input $3$-uniform hypergraphs on $n$ vertices is at least as hard as the 3-OCC-MAX-2-SAT problem for formulas on $n$-variables.
	The latter problem is NP-hard by Theorem \ref{thm:3-OCC-MAX-2-SAT-NP-hard}.
	Let $\varphi$ be a 3-OCC-2-SAT formula on variables $x_1, \dots, x_n$ that contains $m$ clauses, denoted by $C_1, \dots, C_m$.
	We first observe that we may assume, without loss of generality, that each literal in $\varphi$ appears at most twice.
	Indeed, if some variable $x_i$ appears in $\varphi$ three times such that all of its appearances are either as the literal $x_i$ or as $\overline{x_i}$, then by assigning to $x_i$ the values $1$ or $0$, respectively, we satisfy all three clauses in which $x_i$ appears, and thus can remove them from the formula.
	Moreover, by a similar reasoning, we may assume that for each variable $x_i$, both of its literals appear in $\varphi$.
	Now, given formula $\varphi$ as above, we construct a $3$-uniform hypergraph that we denote by $G_{\varphi}$, as follows.
	For each variable $x_i, \, 1 \leq i \leq n$, we put two edges that intersect in $2$ vertices. The edges are put on newly added vertices.
	We denote these edges by $e_{x_i}, e_{\overline{x_i}}$, and they represent the literals $x_i, \overline{x_i}$, respectively.
	Now, for each clause $C = (a \lor b)$ in $\varphi$ (where $a,b$ are literals), we add a configuration of two edges, as follows.
	Let us denote $e_a = \{v_1, v_2, v_3\}, e_b = \{u_1, u_2, u_3\}$ such that
	$e_a \cap e_{\overline{a}} = \{v_1, v_2\}, \, e_b \cap e_{\overline{b}} = \{u_1, u_2\}$.
	Now, if it is the first occurrence of $a$ so far in $\varphi$, add the edge $f_1^a = \{v_1, v_3, u_3\}$, and if it is the second occurrence, add the edge $f_2^a = \{v_2, v_3, u_3\}$.
	Similarly, if it is the first occurrence of $b$, add the edge $f_1^b = \{u_1, u_3, v_3\}$, and if it is the second occurrence, add the edge $f_2^b = \{u_2, u_3, v_3\}$.
	Observe that $(e_a,f_*^a,f_*^b,e_b)$ is a sequence of four edges such that every two consecutive edges intersect in $2$ vertices (see Figure \ref{fig:clause-config} for an illustration).
	The resulting graph is $G_{\varphi}$. Note that $|E(G_{\varphi})| = 2n + 2m = \poly(n)$.
	Also, it is straightforward to verify that the only pairs of edges in $G_{\varphi}$ that intersect in $2$ vertices are $e_{x_i}, e_{\overline{x_i}}, \, 1 \leq i \leq n$, and, for every clause $C = (a \lor b)$, each pair of consecutive edges in the four-edge sequence $(e_a,f_*^a,f_*^b,e_b)$ associated with $C$.
	
	\begin{figure}[h] 
		\centering
		
		\begin{tikzpicture}[scale = 2.25]
		\coordinate (u3) at (0,0.2);
		\coordinate (u1) at (0.5,0.5);
		\coordinate (u2) at (1,0.5);
		\coordinate (u4) at (1.5,0.2);
		
		\foreach \i in {1,2,3,4}
		{
			\draw (u\i) node[fill=black,circle,minimum size=3pt,inner sep=0pt,label=below:$u_\i$] {};
		}
		
		\draw [rounded corners = 18,label=$f$] (-0.44,-0.15) -- (0.5,0.75) -- (1.5,0.4) -- cycle;
		\coordinate (e-pos-xj) at (0.47,0.96);
		\draw (e-pos-xj) node[fill=black,circle,minimum size=0pt,inner sep=0pt,
		label={[font=\small,text=blue]below:$\mathbf{e_{x_j}}$}] {};
		
		\draw [rounded corners = 18] (1.94,-0.15) -- (1,0.75) -- (0,0.4) -- cycle;
		\coordinate (e-neg-xj) at (1.03,0.96);
		\draw (e-neg-xj) node[fill=black,circle,minimum size=0pt,inner sep=0pt,
		label={[font=\small,text=blue]below:$\mathbf{e_{\overline{x_j}}}$}] {};
		
		\coordinate (v3) at (0,-1.2);
		\coordinate (v1) at (0.5,-1.5);
		\coordinate (v2) at (1,-1.5);
		\coordinate (v4) at (1.5,-1.2);
		
		\foreach \i in {1,2,3,4}
		{
			\draw (v\i) node[fill=black,circle,minimum size=3pt,inner sep=0pt,label=above:$v_\i$] {};
		}
		
		\draw [rounded corners = 18,label=$f$] (-0.44,-0.85) -- (0.5,-1.75) -- (1.5,-1.4) -- cycle;
		\coordinate (e-pos-xi) at (0.47,-1.96);
		\draw (e-pos-xi) node[fill=black,circle,minimum size=0pt,inner sep=0pt,
		label={[font=\small,text=blue]above:$\mathbf{e_{x_i}}$}] {};
		
		\draw [rounded corners = 18] (1.94,-0.85) -- (1,-1.75) -- (0,-1.4) -- cycle;
		\coordinate (e-neg-xi) at (1.03,-1.96);
		\draw (e-neg-xi) node[fill=black,circle,minimum size=0pt,inner sep=0pt,
		label={[font=\small,text=blue]above:$\mathbf{e_{\overline{x_i}}}$}] {};
		
		\draw [rounded corners = 18] (0.7,-1.67) -- (-0.55,-0.95) -- (0,0.45) -- cycle;
		\draw [rounded corners = 18] (0.7,0.65) -- (-0.55,0) -- (0,-1.45) -- cycle;
		
		\end{tikzpicture}
		
		\caption{A configuration of clause $(x_i \lor x_j)$.}\label{fig:clause-config}
	\end{figure}
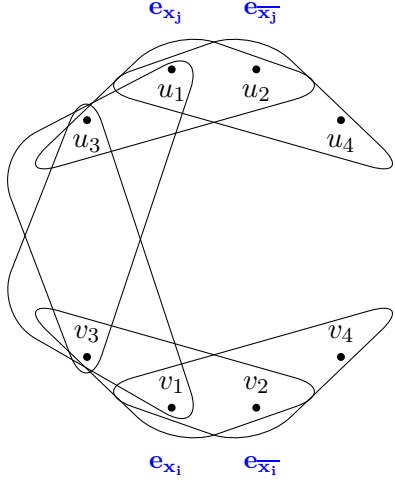
	
	\noindent For a 2-CNF formula $\varphi$, denote by $\ell_{\varphi}$ the maximum number of clauses that can be satisfied in $\varphi$ by some truth assignment to its variables. Our main claim is as follows:	
	\begin{claim}\label{clm:3-OCC-MAX-2SAT-ex-reduction}
		Given a 3-OCC-2-SAT formula $\varphi$ on variables $x_1, \dots, x_n$ and with $m$ clauses,
		it holds that $\ex(G_{\varphi},E_2^{(3)}) = n + \ell_{\varphi}$.
	\end{claim}
	\begin{proof}
		Let $\exbar(G_{\varphi},E_2^{(3)}) = e(G_{\varphi}) - \ex(G_{\varphi},E_2^{(3)})$.
		Recalling that $|E(G_{\varphi})|=2n+2m$, it is enough to show that
		$\exbar(G_{\varphi},E_2^{(3)}) = n+2m-\ell_{\varphi}$.
		We first observe that $\exbar(G_{\varphi},E_2^{(3)}) \geq n+m$. Indeed, in order to turn $G_{\varphi}$ into an $E_2^{(3)}$-free graph, we need to remove at least one of the edges $e_{x_i}, e_{\overline{x_i}}$ for every $1 \leq i \leq n$, and, for every clause $C = (a \lor b)$, we need to remove at least one of the
		middle edges of the four-edge sequence $(e_a,f_*^a,f_*^b,e_b)$ associated with $C$.
		Let $f : \{x_1, \dots, x_n\} \rightarrow \{0,1\}$ be a truth assignment which satisfies $\ell_{\varphi}$ clauses in $\varphi$.
		We now define a set of edges $E \subseteq E(G_{\varphi})$, as follows. For every $i \in [n]$, if $f(x_i) = 1$ then add the edge $e_{x_i}$ to $E$, and otherwise add the edge $e_{\overline{x_i}}$ to $E$.
		Now, for every clause $C = (a \lor b)$ in $\varphi$, if $e_a, e_b \notin E$, then add to $E$ both edges $f_*^a,f_*^b$ that are associated with $C$; otherwise, if $e_a \in E$ then add edge $f_*^b$ to $E$, and otherwise add edge $f_*^a$ to $E$.
		Now, as for every clause $C = (a \lor b)$ that is satisfied by $f$ we add to $E$ only one of the edges $f_*^a,f_*^b$ (since $e_a \in E$ or $e_b \in E$, by our construction of $E$), we have that
		$|E| = n + \ell_{\varphi} + 2(m-\ell_{\varphi}) = n+2m-\ell_{\varphi}$.
		We now claim that $G_{\varphi} - E$ is $E_2^{(3)}$-free.
		Indeed, for every $i \in [n]$, we have removed one of the edges $e_{x_i}, e_{\overline{x_i}}$, and it is easy to see that for every clause $C = (a \lor b)$, by removing $E$ we leave no two consecutive edges
		in the sequence $(e_a,f_*^a,f_*^b,e_b)$ associated with $C$. As we previously argued,
		no other two edges in $G_{\varphi}$ intersect in $2$ vertices, and so $G_{\varphi} - E$ is $E_2^{(3)}$-free. We conclude that $\exbar(G_{\varphi},E_2^{(3)}) \leq |E| = n+2m-\ell_{\varphi}$.
		
		In the other direction, assume by contradiction that there exists a set of edges $E \subseteq E(G_{\varphi})$ with $|E| < n+2m-\ell_{\varphi}$ such that $G_{\varphi} - E$ is $E_2^{(3)}$-free. If both $e_{x_i}, e_{\overline{x_i}} \in E$ for some $i \in [n]$, then assuming, without loss of generality, that the literal $x_i$ occurs only once in $\varphi$, remove $e_{x_i}$ form $E$ and add instead the edge $f_1^{x_i}$. Clearly, we did not increase the size of $E$, and it is easy to see\footnote{This is true since $e_{x_i}$ intersects in $2$ vertices only with the edges $e_{\overline{x_i}}, f_1^{x_i}$, and both are in $E$.} that $G_{\varphi} - E$ is $E_2^{(3)}$-free also with the updated set $E$.
		Now, define a truth assignment $f$ as follows: For every $i \in [n]$, set $f(x_i) = 1$ if
		$e_{x_i} \in E$, and set $f(x_i) = 0$ otherwise.
		As for every $i \in [n]$, exactly one of the edges $e_{x_i}, e_{\overline{x_i}}$ is in $E$, and as for every clause $C = (a \lor b)$, the set $E$ must contain at least one of the edges $f_*^a,f_*^b$ associated with $C$, then since $|E| < n+2m-\ell_{\varphi}$, it follows that the number of clauses $C = (a \lor b)$ for which $E$ contains exactly one of the edges $f_*^a,f_*^b$ is at least $\ell_{\varphi}+1$. For each such clause $C = (a \lor b)$ we have that $e_a \in E$ or $e_b \in E$ (since $G_{\varphi} - E$ is $E_2^{(3)}$-free), and thus, by the definition of $f$ (and using the fact that for every $i \in [n]$, exactly one of the edges $e_{x_i}, e_{\overline{x_i}}$ is in $E$), we have that $f$ satisfies $C$, and so $f$ satisfies at least $\ell_{\varphi}+1$ clauses in $\varphi$, a contradiction.
		We conclude that $\exbar(G_{\varphi},E_2^{(3)}) \geq n+2m-\ell_{\varphi}$, completing the proof of the claim.
	\end{proof}
	
	\noindent The above claim establishes the desired reduction and completes the proof of the lemma.
\end{proof}

\begin{remark}\label{rem:(3,1)-also-hard}
	On can easily verify that essentially the same reduction we have established in Lemma \ref{lem:(3,2)-hardness}, changing only slightly the construction of the graph $G_{\varphi}$, can be used to prove also that computing $\ex(G,E_1^{(3)})$ is NP-hard for $3$-uniform hypergraphs.
	Indeed, the only difference in the construction of $G_{\varphi}$, compared to the one in the proof of Lemma \ref{lem:(3,2)-hardness}, is that now, for every every clause $C = (a \lor b)$, we shall add $3$ new vertices that we denote by $z_1^{a,b}, z_2^{a,b}, z_3^{a,b}$, and then add a configuration of two edges $f^a, f^b$, where $f^a$ will contain $z_1^{a,b}, z_2^{a,b}$ and a vertex from $e_a \setminus e_{\overline{a}}$, while $f^b$ will contain $z_2^{a,b}, z_3^{a,b}$ and a vertex from $e_b \setminus e_{\overline{b}}$, forming a loose path $(e_a,f^a,f^b,e_b)$ associated with $C$.
	The rest of the proof remains valid also for this construction of $G_{\varphi}$.
	For the sake of clarity and brevity, we stated and proved in Lemma \ref{lem:(3,2)-hardness} only the case
	$k=3, \, t=2$.
\end{remark}

We are now ready to prove Theorem \ref{thm:two-intersecting-edges-hardness}.
\begin{proof}[Proof of Theorem \ref{thm:two-intersecting-edges-hardness}]
	Let $k \geq 3, \, k > t \geq 1$ be integers.
	We first deal with the case $t=1$. By Remark \ref{rem:(3,1)-also-hard}, we know that computing
	$\ex(G,E_1^{(3)})$ is NP-hard for $3$-uniform hypergraphs, and so we may assume that $k > 3$.
	Now, starting with an $n$-vertex $3$-graph $G$ and $E_1^{(3)}$, and applying (the reduction in) Proposition \ref{prop:ex(G,k,t)<=ex(G',k+1,t)} $k-3$ times, we get that computing
	$\ex(G',E_1^{(k)})$ is NP-hard for input $k$-uniform hypergraphs, as required.
	
	We now assume that $t \geq 2$.
	By Lemma \ref{lem:(3,2)-hardness}, we know that computing
	$\ex(G,E_2^{(3)})$ is NP-hard for $3$-uniform hypergraphs, and so we may assume $(k,t) \neq (3,2)$.
	Now, starting with an $n$-vertex $3$-graph $G$ and $E_2^{(3)}$, by first applying (the reduction in) Proposition \ref{prop:ex(G,k,t)<=ex(G',k+1,t)} $k-t-1$ times, and then applying (the reduction in) Proposition \ref{prop:ex(G,k,t)<=ex(G',k+1,t+1)} $t-2$ times, we get that computing
	$\ex(G',E_t^{(k)})$ is NP-hard for input $k$-uniform hypergraphs, as required.
	This completes the proof of the theorem.
\end{proof}

\begin{remark}\label{rem:r-edges-intersect-in-t-vert}
	We note that the hardness result in Theorem \ref{thm:two-intersecting-edges-hardness} can
	be easily generalized to $r$ intersecting edges for every $r \geq 2$. Namely, given integers
	$k \geq 3, \, k > t \geq 1, \, r \geq 2$, we denote by $E_t^{(k)}(r)$ a set of $r$ edges of size $k$ that intersect in exactly $t$ (same) vertices. One can show that for every $r \geq 2$, computing $\ex(G,E_t^{(k)}(r))$ is NP-hard for  $k$-uniform hypergraphs. The proof is by induction on $r$, where the base case $r=2$ was established in Theorem \ref{thm:two-intersecting-edges-hardness}. We now assume that computing $\ex(G,E_t^{(k)}(r))$ is NP-hard for some $r \geq 2$, and show that computing $\ex(G,E_t^{(k)}(r+1))$ is at least as hard as computing $\ex(G,E_t^{(k)}(r))$ for input $k$-uniform hypergraphs. Given an $n$-vertex graph $G$, we construct a graph $G'$, as follows. For every $t$-set of vertices $T \in \binom{V(G)}{t}$, add to $G$ new $k-t$ vertices that we denote by $z_1^T, \dots, z_{k-t}^T$, and then add edge $e_T := T \cup \{z_1^T, \dots, z_{k-t}^T\}$. The resulting graph is $G'$. Note that $v(G') = n + \binom{n}{t} \cdot (k-t) = \textup{poly}(n), \,
	e(G') = e(G) + \binom{n}{t} = \textup{poly}(n)$.
	Now, observing that no two of the newly added edges intersect in $t$ vertices, and also that
	if $F$ is an $E_t^{(k)}(r+1)$-free subgraph of $G'$ then at least
	$e(F[V(G)]) - \ex(G,E_t^{(k)}(r))$ of the newly added edges are missing from $F$,
	it is straightforward to verify that $\ex(G',E_t^{(k)}(r+1)) = \ex(G,E_t^{(k)}(r))  + \binom{n}{t}$.
\end{remark}

\end{document}